\documentclass[12pt]{amsart} 

\usepackage{graphicx,amssymb,colordvi,textcomp,latexsym} 

\usepackage{tikz}

\usepackage[title]{appendix}

\usepackage{tikz}
\usetikzlibrary{arrows,shapes, positioning, matrix, patterns, decorations.pathmorphing, arrows.meta,automata}

\usepackage{color}
\usepackage{colortbl}
\usepackage{soul}
\usepackage{array}
\usepackage[hidelinks]{hyperref}
\usepackage{amsfonts}
\usepackage{amsmath}
\usepackage{amsthm}
\usepackage{adjustbox}
\usetikzlibrary{arrows,shapes, positioning, matrix, patterns, decorations.pathmorphing}
\usepackage{todonotes}
\usepackage{enumitem}

\newcommand{\R}{\mbox{$\mathbb{R}$}}

\newcommand{\rset}[2]{\left\lbrace\, #1 \,\left|\ #2\right.\right\rbrace}

\newcommand{\set}[2]{\rset{#1}{#2}}

\newcommand{\sset}[1]{\left\lbrace #1\right\rbrace}
\newcommand{\tsset}[1]{\big\lbrace #1\big\rbrace}

\newcommand{\Hc}{\mathcal{H}}
\newcommand{\Dc}{\mathcal{D}}

\newtheorem{lemma}{Lemma}[section]

\newtheorem{prop}[lemma]{Proposition}
\newtheorem{thm}[lemma]{Theorem}

\theoremstyle{definition}
\newtheorem{Def}[lemma]{Definition}
\newtheorem{exam}[lemma]{Example}
\newtheorem{exams}[lemma]{Examples}

\theoremstyle{remark}
\newtheorem{rem}[lemma]{Remark}

\newcommand{\simI}{\mathrel{\sim_I}}
\newcommand{\simIo}{\mathord{\sim_I}}

\title[Coupled Cell Hypernetworks]{Network Dynamics with Higher-Order Interactions: Coupled Cell Hypernetworks for Identical Cells and Synchrony}
\author{Manuela Aguiar}
\address{Manuela Aguiar, 
Faculdade de Economia, Centro de Matem\'atica, Universidade do Porto, Rua Dr Roberto Frias, 4200-464 Porto, Portugal.}
\email{maguiar@fep.up.pt}
\author{Christian Bick}
\address{Christian Bick, 
Department of Mathematics, Vrije Universiteit Amsterdam, De Boelelaan 1111, 1081 HV Amsterdam, the Netherlands;
Department of Mathematics, University of Exeter, Exeter EX4 4QF, United Kingdom;
Institute for Advanced Study, Technical University of Munich, Lichtenbergstr.~2, 85748 Garching, Germany;
Mathematical Institute, University of Oxford, Oxford OX2 6GG, United Kingdom
}
\email{c.bick@vu.nl}
\author{Ana Dias}
\address{Ana Dias, 
Departamento de Matem\'atica, Faculdade de Ci\^encias, Centro de Matem\'atica, Universidade do Porto, Rua do Campo Alegre, 687, 4169-007 Porto, Portugal}
\email{apdias@fc.up.pt}

\begin{document}

\allowdisplaybreaks

\maketitle

\newcommand{\cG}{\mathcal{G}}
\newcommand{\cH}{\mathcal{H}}
\newcommand{\cT}{\mathcal{T}}

\newcommand{\N}{\mathbb{N}}

\newcommand{\ud}{\mathrm{d}}
\newcommand{\udi}{\,\ud}

\newcommand{\CB}[1]{\textcolor{blue}{#1}}
\newcommand{\AD}[1]{\textcolor{purple}{#1}}
\newcommand{\MA}[1]{\textcolor{orange}{#1}}

\begin{abstract} 
Network interactions that are nonlinear in the state of more than two nodes---also known as higher-order interactions---can have a profound impact on the collective network dynamics. Here we develop a coupled cell hypernetwork formalism to elucidate the existence and stability of (cluster) synchronization patterns in network dynamical systems with higher-order interactions. 
More specifically, we define robust synchrony subspace for coupled cell hypernetworks whose coupling structure is determined by an underlying hypergraph and describe those spaces for general such hypernetworks.
Since a hypergraph can be equivalently represented as a bipartite graph between its nodes and hyperedges, we relate the synchrony subspaces of a hypernetwork to balanced colorings of the corresponding incidence digraph.
\end{abstract}


\section{Introduction}

Coupled dynamical processes are ubiquitous in the world and can often be modeled by systems of ordinary differential equations (ODEs). The coupled cell network formalism developed by Golubitsky, Stewart and collaborators~\cite{SGP03, GST05} and Field~\cite{F05} captures the network interactions by a directed graph~$\cG$ to elucidate how the network structure shapes the collective dynamics. More precisely, let~$V=\R^d$ for some~$d\in\N$ denote the state space of each cell $i\in\sset{1, \dotsc, n}$. In a classical coupled cell system, the evolution state~$x_i$ of cell~$i$ is determined by an interaction function~$f:V^q\to V$. If, for example,
\begin{equation}\label{eq:CoupCellEx}
\dot x_i := \frac{\ud x_i}{\ud t} = f(x_i; x_j, x_k, x_l)
\end{equation}
then $(j,i), (k,i), (l,i)$ are the edges with head~$i$ of~$\cG$ since, for any~$f$, the evolution of cell~$i$ depends on the cells~$j,k,l$. The main questions regarding coupled cell networks relate to how the network structure influences the dynamics and bifurcations of the coupled cell system without making specific assumptions on~$f$. By contrast, in many applications the links in the networks have associated numerical values called \emph{weights} to represent, for example, the strength or the signal of the connection between the nodes associated with the edges.
These can be realized as coupled cell networks with \emph{additive input structure}; cf.~\cite{F15, BF17, ADF17, AD18}. Consider the graph~$\cG$ associated with~\eqref{eq:CoupCellEx} and let $(w_{ij})\in\R^{n\times n}$ be a weight matrix. For $h:V\to V$ and $g:V\times V\to V$, cell~$i$ of the corresponding coupled cell network with additive coupling structure evolves according to
\begin{equation}\label{eq:CoupCellAddEx}
\dot x_i := h(x_i) + w_{ij}g(x_i; x_j) + w_{ik}g(x_i; x_k)  + w_{il}g(x_i; x_l),
\end{equation}
where~$g$ determines the \emph{pairwise interactions} between cells. In this restricted framework, adding and removing edges is natural by adjusting the corresponding weights. Networks of Kuramoto phase oscillators and pulse coupled systems are examples of coupled cell systems with additive input structure.

Note that the complexity of the interactions differ in traditional coupled cell networks~\eqref{eq:CoupCellEx} and those with additive coupling structure~\eqref{eq:CoupCellAddEx}. While the former allows for generic, nonlinear interactions between all the input nodes through~$f$, additive coupling structure only allows for interactions between pairs of nodes. Recent research has highlighted the dynamical importance of nonpairwise interactions between nodes; cf.~\cite{Battiston2020,Bick2021} for reviews. For example, in networks that describe the competitive interactions between species, one has to take into account how the interaction between two species is modulated by a third species (a triplet interaction) to explain the competition dynamics. Similarly, incorporating nonpairwise interactions in phase oscillator networks exhibits dynamics that would not arise in standard Kuramoto-type equations with pairwise interactions~\cite{AR16, BAR16}.

In this work, we introduce a new class of coupled cell networks---\emph{coupled cell hypernetworks}---whose structure is determined  by a \emph{(directed) hypergraph}. A hypergraph is a generalization of a graph in which a hyperedge can join any number of nodes, that is, the \emph{directed hyperedges} are from a set of~$k$ nodes (cells) to a set of~$l$ nodes (cells). 
This coupling structure captures that the evolution of each of the~$l$ cells depends (typically nonlinearly) on an interaction involving a set of~$k$ cells. 
Directed hypergraphs are used to model problems arising in, for example, operations research, computer science and discrete mathematics, to describe relationships  between two sets of objects. See for example Ausiello and Laura~\cite{AL17} and references therein. See, also, Johnson {\it et al.}~\cite{JI07}, Kim~{\it et al.}~\cite{KHZ14} and Johnson~\cite{J16}. 
We shall remark that in some literature, as for example in Sorrentino~\cite{S12}, the terminology of hypernetwork is used, not to denote a hypergraph, as in our case here, but to denote a graph that has more than one edge type, that is, with more than one adjacency matrix. We illustrate our setup in an example.

\begin{exam} \label{ex:first}
Consider the following system of ODEs on $n=6$ state variables $x_i$, $i\in\{1,\dotsc,n\}$:
\begin{subequations}\label{ex:OO}
\begin{align}
\dot{x}_1 &= f(x_1) + Q_2(x_1;x_5,x_2)\\
\dot{x}_2 &= f(x_2) + Q_1(x_2;x_2)\\
\dot{x}_3 &= f(x_3) + Q_1(x_3;x_4) + Q_2(x_3;x_4,x_6)\\
\dot{x}_4 &= f(x_4) + Q_1(x_4;x_2)\\
\dot{x}_5 &= f(x_5) + Q_2(x_5;x_4,x_6) \\
\dot{x}_6 &= f(x_6) + Q_2(x_6;x_1,x_2),
\end{align}
\end{subequations}
where $f:\, V \to V$, $Q_1:\, V^2 \to V$,  $Q_2:\, V^3 \to V$ are smooth functions. 
Assume that~$Q_2$ is symmetric under permutation of the last two coordinates, that is, $Q_2(y;z,w) = Q_2(y;w,z)$ for all $y,z,w \in V$.
We might interpret this system as a coupled cell system with form consistent with a hypergraph~$\Hc$ shown on the left of Figure~\ref{figure_example1}: Each node of the hypergraph represents a cell, and each hyperedge represents an interaction from a cell---or a group of cells---to a cell or a group of cells.
The state of cell~$i$ is determined by~$x_i\in V$ and its evolution by the corresponding differential equation; in the following we write $x=(x_1, \dotsc, x_n)$ for the state vector. 
The coupling functions~$Q_1$ and~$Q_2$ determine the influence of one or two cells, respectively, onto another cell along the directed hyperedges.

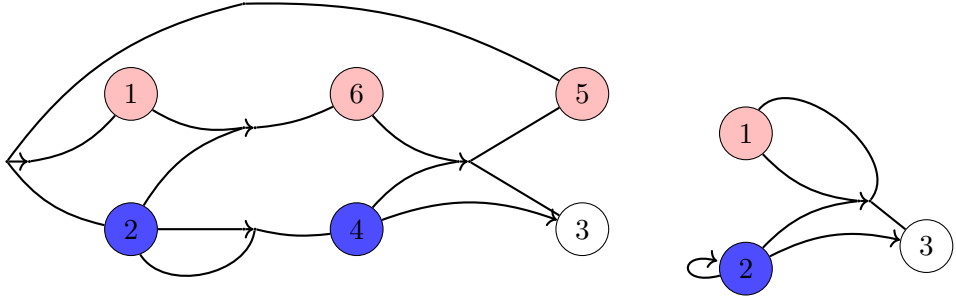
\begin{figure}
\begin{center}
\begin{tabular}{ll}
\begin{tikzpicture}
 [scale=.15,auto=left, node distance=1.5cm, 
 ]
 \node[fill=pink,style={circle,draw}] (n1) at (4,6)  {\small{1}};
 \node[fill=blue!70,style={circle,draw}] (n2) at (4,-6) {\small{2}};
 \node[fill=pink,style={circle,draw}] (n6) at (24,6) {\small{6}};
   \node[fill=blue!70,style={circle,draw}] (n4) at (24,-6)  {\small{4}};
 \node[fill=pink,style={circle,draw}] (n5) at (44,6)  {\small{5}};
  \node[fill=white,style={circle,draw}] (n3) at (44,-6)  {\small{3}};
 
 \node [fill=black,style={circle,scale=0.1}] (e1) at (14,3) { };
 \node [fill=black,style={circle,scale=0.1}] (e11) at (15,3) { };
 \node [fill=black,style={circle,scale=0.1}] (e5) at (14,-6) { };
 \node [fill=black,style={circle,scale=0.1}] (e55) at (15,-6) { };
    \node [fill=black,style={circle,scale=0.1}] (e2) at (33,0) { };
     \node [fill=black,style={circle,scale=0.1}] (e22) at (34,0) { };
     \node [fill=black, style={circle, scale=0.1}] (e3) at (-7,0) { };
      \node [fill=black, style={circle, scale=0.1}] (e33) at (-5,0) { };
      \node [fill=black,style={circle,scale=0.1}] (e4) at (14,14) { };

 %
\path
        (n1) [-]  edge[bend right=20,thick]  node { } (e1)
        (n2) [-]  edge[bend left=20,thick]  node { } (e1)
         (e1) [->] edge[thick] node {}   (e11)   
        (e11) [-] edge[bend right=10,thick]      node  {}  (n6)
        (n2) [-] edge[bend left=0, thick]  node {}   (e5)
             (e5) [->] edge[thick] node {}   (e55)  
              (e55) [-] edge[bend right=10,thick]      node  {}  (n4)
         (e55)  [-]  edge[bend left=70,thick] node {} (n2)
        (n4) [->] edge[bend left=20, thick]  node {}   (n3)      
        (n6) [-]  edge[bend right=20,thick]  node { } (e2)
        (n4) [-]  edge[bend left=20,thick]  node { } (e2)
         (e2) [->] edge[thick] node {}   (e22)   
          (e22) [-] edge[thick] node {}   (n5)   
              (e22) [-] edge[thick] node {}   (n3)   
  (n5) [-] edge[bend right=15,thick]  node {} (e4)
        (e4) [-] edge[bend right=20,thick] node { } (e3)
         (n2) [-]  edge[bend left=20,thick] node { } (e3)
        (e3)  [->]  edge[thick] node {} (e33)
        (e33)  [-]  edge[bend right=20,thick] node {} (n1); 
 \end{tikzpicture} & 
 \quad
 \begin{tikzpicture}
 [scale=.15,auto=left, node distance=1.5cm, 
 ]
 \node[fill=pink,style={circle,draw}] (n1) at (4,6)  {\small{1}};
 \node[fill=blue!70,style={circle,draw}] (n2) at (4,-6) {\small{2}};
  \node[fill=white,style={circle,draw}] (n3) at (20,-4)  {\small{3}};
     \node [fill=black,style={circle,scale=0.1}] (e3) at (14,0) { };
        \node [fill=black,style={circle,scale=0.1}] (e33) at (15,0) { };
\path
        (n1) [-]  edge[bend right=20,thick] node { } (e3)
         (n2) [-]  edge[bend left=20,thick] node { } (e3)
         (e3) [->] edge[thick] node {} (e33)
        (e33)  [-]  edge[bend right=90,thick] node {} (n1)   
         (n2)  [->]  edge[loop left=90,thick] node {} (n2)
         (n2) [->] edge[bend left=20, thick]  node {}   (n3) 
       (e33)  [-]  edge[thick] node {} (n3);    
       
 \end{tikzpicture} 
 \end{tabular}
 \end{center}
 \caption{Examples of two directed hypergraphs: Nodes (cells) are shown as circles and directed hyperedges as arrows that can have multiple nodes in the tail (lines from multiple nodes leading up to the arrow) and multiple nodes in the head (lines from the arrow to the receiving cells). Assume all hyperedges have weight~$1$. The shading of the nodes/cells corresponds to the synchrony pattern described in Example~\ref{ex:first}.} \label{figure_example1}
\end{figure}

Now consider subsets of the phase space where cells are \emph{synchronized}, that is, there are distinct cells whose states take the same value; sometimes this is also referred to as \emph{cluster synchronization}. Some synchronization patterns are \emph{robust}, that is, they are dynamically invariant subsets of the phase space for any coupling functions. In our example, consider the set $\set{x}{x_1 = x_6 = x_5,\, x_2 = x_4}$, where cells~$1$,~$6$,~$5$ as well as~$2$,~$4$ are synchronized. Note that this set is invariant under the flow of the above equations and the dynamics restricted to this space are given by
\begin{subequations}\label{ex:OOQuot}
\begin{align}
\dot{x}_1 &= f(x_1) + Q_2(x_1;x_1,x_2)\\
\dot{x}_2 &= f(x_2) + Q_1(x_2;x_2) \\
\dot{x}_3 &= f(x_3) + Q_1(x_3;x_2) + Q_2(x_3;x_1,x_2).
\end{align}
\end{subequations}
These are again dynamical equations that can be interpreted as a coupled cell hypernetwork; one underlying hypergraph is shown in Figure~\ref{figure_example1} on the right.

This illustrates some of the main questions we will address here: Given a set of dynamical equations, such as~\eqref{ex:OO}, what is the underlying hypergraph? Given a hypergraph~$\Hc$ and an associated coupled cell hypernetwork, how can we identify the robust synchrony subspaces? Given a robust synchrony subspace, how can we describe the dynamics on the robust synchrony subspace as a coupled cell hypernetwork and how does this relate to the original hypergraph~$\Hc$?
\hfill $\Diamond $
\end{exam}

The main contribution of this paper is to develop the framework of coupled cell hypernetworks and apply this framework to analyze the existence and stability of synchrony in hypernetwork dynamical systems. While the dynamical equations are similar to those in~\cite{Mulas2020,Salova2021a,Salova2021}, we explicitly discuss the role of the interaction functions~$Q_k$. Placed within the language of coupled cell networks, our approach allows to use the general ideas developed in~\cite{AD18} for the analysis of network dynamical systems with higher-order interactions. Specifically, the manuscript is organized as follows. Section~\ref{Sec:hypergraph} reviews some definitions and notation on directed weighted hypergraphs. The coupled cell hypernetwork formalism for coupled differential equations is introduced in Section~\ref{Sec:hypernet}. In Section~\ref{sec:sync} we define robust synchrony subspace for 
hypernetworks, describe those spaces for general hypernetworks and we relate  them to the balanced colorings of the corresponding incidence digraph. In Section~\ref{sec:examples} we discuss a class of hypernetworks where we can relate stability of equilibria taking into account the nonpairwise interactions. We see already for this class of examples that the nonpairwise terms cannot be  disregarded. We finish with Section~\ref{sec:discussion} discussing the main points presented in this work and some questions that arise naturally.

\section{Preliminaries on directed hypergraphs}\label{Sec:hypergraph}

In this section, we recall some notation and definitions on directed hypergraphs; see, for example,~\cite{GLP93}. An hypergraph is a generalization of a graph where the graph edges are replaced by hyperedges that can join any number of nodes. In contrast to traditional directed hypergraphs, we allow for the tails to be multisets, i.e., a set that can contain an element more than once. Let $\#A$ denote the cardinality of a (multi)set~$A$.

\begin{Def}
A \emph{directed hypergraph} $\cH =(C,E)$  consists of a (finite) set of \emph{nodes}~$C$ and a set of \emph{directed hyperedges}~$E$. A \emph{directed hyperedge~$e$} is a pair $\left(T(e),H(e)\right)$, where the \emph{tail}~$T(e)$ of~$e$ is a multiset of elements of~$C$ and the \emph{head}~$H(e)$ of~$e$ is a subset of $C$; we assume that both $T(e)$ and $H(e)$ are nonempty. 
\hfill $\Diamond$
\end{Def}
 
Note that, a directed hypergraph where any hyperedge~$e$ satisfies the conditions $\# T(e) = \# H(e) = 1$ is a standard \emph{directed graph}.
 
In the above definition of directed hypergraph, we do not exclude the situation of having hyperedges $e$ where the tail multiset $T(e)$ has repetition of nodes. This fact is due to the association of hypergraphs with coupled cell hypernetworks and it will be clarified in Section~\ref{Sec:hypernet}.
 
If $\cH =(C,E)$ is a hypergraph, we also write~$C(\cH)=C$ or $E(\cH)=E$ to denote the set of nodes and hyperedges, respectively.
 
\begin{exam} 
The directed hypergraph $\cH=(C,E)$  in Figure~\ref{h_net_1_v2} has node set $C = \{ 1, 2, 3, 4\}$ and hyperedge set  
\[E = \left\{e_1 = \left( \{1,2\}, \{3,4\}\right),\ e_2 = \left( \{1\}, \{4\} \right), e_3 =   \left( \{1\}, \{1\}\right)
\right\}.\] 
\hfill $\Diamond $
\end{exam}

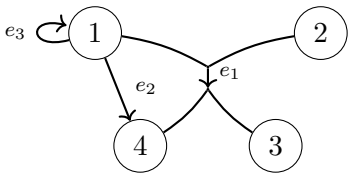
\begin{figure}
\begin{center}
\begin{tikzpicture}
 [scale=.15,auto=left, node distance=1.5cm, 
 ]
 \node[fill=white,style={circle,draw}] (n1) at (4,0) {\small{1}};
  \node[fill=white,style={circle,draw}] (n2) at (24,0) {\small{2}};
 \node[fill=white,style={circle,draw}] (n3) at (20,-10)  {\small{3}};
 \node[fill=white,style={circle,draw}] (n4) at (8,-10)  {\small{4}};
 \node [fill=black,style={circle,scale=0.1}] (e1) at (14,-3) { };
 \node [fill=black,style={circle,scale=0.1}] (e11) at (14,-5) { };
\path
        (n1) [-]  edge[bend left=10,thick]  node { } (e1)
        (n2) [-]  edge[bend right=10,thick]  node { } (e1)
        (e1) [->] edge[thick] node [near start] {{\tiny $e_1$}} (e11)
        (e11) [-] edge[bend right=10,thick] node {} (n3) 
         (e11) [-] edge[bend left=10, thick] node {} (n4) 
       (n1) [->] edge[thick] node [near end]  {{\tiny $e_2$}} (n4)  
       (n1)  [->]  edge[loop left=90,thick] node {{\tiny $e_3$}} (n1)
        ;
 \end{tikzpicture} 
\end{center}
\caption{A directed hypergraph with four nodes and three hyperedges labeled by $e_1, e_2, e_3$.}\label{h_net_1_v2}
\end{figure}

\begin{Def} 
Consider a directed hypergraph $\Hc=(C,E)$ with set of~$n$ nodes $C=\{1, \dotsc, n\}$ and set of~$m$ directed hyperedges $E=\{e_1, \dotsc, e_m\}$. Let $w: E \to \R$ be the weight function that associates a weight $w_j$ to each hyperedge $e_j$, $j=1,\dotsc,m$. The \emph{weight matrix} $W\in\R^{n \times m}$  of~$\cH$ is the $n \times m$ matrix, where the $ij$th entry is the weight~$w_j$ of the hyperedge~$e_j$ if node~$i$ belongs to the head of the directed hyperedge~$e_j$, and~$0$ otherwise. A \emph{weighted directed hypergraph~$(\Hc,W)$} consists of~$\Hc$ and a weight matrix~$W$. 
\hfill $\Diamond$
\end{Def}

Note that the definition of weight matrix of a directed hypergraph is distinct from that of  the weighted adjacency matrix of a standard $n$-node directed graph, which is the $n \times n$ matrix, where the $ij$th entry is the weight~$w_{ij}$ of the directed edge from node $j$ to node~$i$ if there is a directed edge from~$j$ directed to node~$i$, and~$0$ otherwise.

\begin{exam}
Consider the weighted directed graph~$\cG$ with nodes $C(\cG)=\{1,2,3,4\}$ and edges \[E(\cG)=\{ \left(\{1\},\{3\}\right),\, \left(\{1\},\{4\}\right),\, \left(\{2\},\{3\}\right),\, \left(\{2\},\{4\}\right)\}\] on the left of Figure~\ref{net1}. The weighted adjacency matrix of the graph~$\cG$ is:  
\[\left[
\begin{array}{cccc}
0 & 0 & 0  & 0 \\
0 & 0 & 0  & 0 \\
a & b & 0 & 0\\
c & d & 0 & 0
\end{array}
\right]\, .
\]
Consider now the weighted directed hypergraph on the right of  Figure~\ref{net1} with two hyperedges: $e_1 = \left( \{1,2\}, \{3\}\right)$ and $e_2 =\left( \{1,2\}, \{4\} \right)$.  The corresponding weight matrix is:  
\[
\left[ 
\begin{array}{cc}
0        & 0 \\
0        & 0\\
a + b & 0 \\
0       & c + d 
\end{array}
\right].
\]
\hfill $\Diamond $
\end{exam}

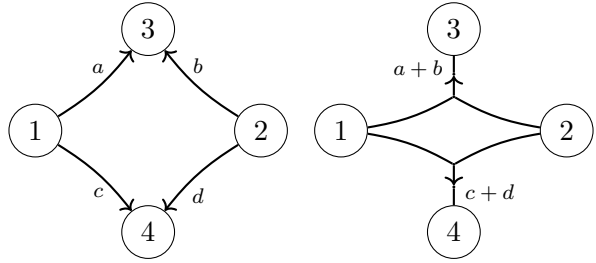
\begin{figure}
\begin{center}
\begin{tabular}{cc} 
\begin{tikzpicture}
 [scale=.15,auto=left, node distance=1.5cm, 
 ]
 \node[fill=white,style={circle,draw}] (n1) at (4,0) {\small{1}};
  \node[fill=white,style={circle,draw}] (n2) at (24,0) {\small{2}};
 \node[fill=white,style={circle,draw}] (n3) at (14,9)  {\small{3}};
 \node[fill=white,style={circle,draw}] (n4) at (14,-9)  {\small{4}};
 \draw[->, thick] (n1) edge[bend right=10,thick]  node  [above=1pt] {{\tiny $a$}} (n3); 
 \draw[->, thick] (n2) edge[bend left=10,thick] node  [above=1pt] {{\tiny $b$}} (n3); 
\draw[->, thick] (n1) edge[bend left=10,thick] node  [below=1pt]  {{\tiny $c$}} (n4); 
\draw[->, thick] (n2) edge[bend right=10,thick] node  [below=1pt]  {{\tiny $d$}} (n4) ; 
\end{tikzpicture}  
&
\begin{tikzpicture}
 [scale=.15,auto=left, node distance=1.5cm, 
 ]
 \node[fill=white,style={circle,draw}] (n1) at (4,0) {\small{1}};
  \node[fill=white,style={circle,draw}] (n2) at (24,0) {\small{2}};
 \node[fill=white,style={circle,draw}] (n3) at (14,9)  {\small{3}};
 \node[fill=white,style={circle,draw}] (n4) at (14,-9)  {\small{4}};
 \node [fill=black,style={circle,scale=0.1}] (e1) at (14,3) { };
  \node [fill=black,style={circle,scale=0.1}] (e11) at (14,5) { };
  \node [fill=black,style={circle,scale=0.1}] (e2) at (14,-3) { };
   \node [fill=black,style={circle,scale=0.1}] (e22) at (14,-5) { };
\path
        (n1) [-]  edge[bend right=10,thick] node { } (e1)
        (n2) [-]  edge[bend left=10, thick] node { } (e1)
        (e1) [->] edge[thick] node  {} (e11)
        (e11) [-] edge[thick] node [near start] {{\tiny $a+b$}} (n3) 
        (n1) [-]  edge[bend left=10,thick]  node { } (e2)
        (n2) [-]  edge[bend right=10,thick] node { } (e2)
        (e2) [->] edge[thick] node {}   (e22) 
        (e22) [-] edge[thick] node  [near start]  {{\tiny $c+d$}} (n4)      
        ;
 \end{tikzpicture} 
\end{tabular}
\end{center}
\caption{(Left) A weighted directed graph  with four nodes and four edges. (Right) A weighted directed hypergraph with four nodes and two hyperedges.}\label{net1}
\end{figure}

To every directed hypergraph~$\cH$ can be associated a bipartite digraph~$\Dc_{\cH}$, called the \emph{incidence digraph}, \emph{Levi digraph}, or \emph{K{\"o}nig digraph} of~$\cH$, whose nodes are the nodes and the hyperedges of~$\cH$; see, for example,~\cite{AS2021}. Here, we generalize this concept to weighted directed hypergraphs (where the tails of the hyperedges can be multisets).

\begin{Def} \label{def:bidigraph}
Consider a weighted directed hypergraph $(\Hc, W)$ with the set of $n$ nodes $C=\{1, \dotsc, n\}$ and a set of $m$ directed hyperedges $E=\{e_1, \dotsc, e_m\}$.  Let $m^j_i$ be the multiplicity of the node $i$ in the tail multiset $T(e_j)$. The {\em weighted incidence digraph} $\Dc_{\cH}$ of $\cH$ is the weighted bipartite digraph with node set $C \cup E$ and edges such that: there is a directed edge from node $i$ to the hyperedge $e_j$ with weight $m^j_i$ if and only if $i \in T(e_j)$; there is a directed edge with weight $w_j$ from the hyperedge $e_j$ to the node $i$ if and only if $i \in H(e_j)$.
\hfill $\Diamond$
\end{Def}

The adjacency matrix $A_{{\Dc_{\cH}}}$ of the weighted incidence digraph $\Dc_{\cH}$ associated with a weighted directed hypergraph $\cH$ has the block structure
\[
A_{\small{\Dc_{\cH}}} =
\left[
\begin{array}{c|c}
0_{n \times n} & W \\
\hline
T & 0_{m \times m}
\end{array}
\right],
\]
where 
$W\in M_{n \times m}(\R)$ is the weight matrix for~$\cH$ and the matrix $T \in M_{m \times n}(\R)$ 
describes the multiplicities of the nodes in the tail multisets of the hyperedges of $\cH$.

\begin{exam}
Consider  the directed hypergraph $\Hc=(C,E)$ on the left in Figure~\ref{figure_example1}. Thus $C = \{ 1, \dotsc, 6\}$ and 
\begin{align*}
e_1 &= \left( \{ 2,5\}, \{1 \} \right), & e_2 &= \left( \{ 2\}, \{2,4 \} \right), & e_3 &= \left( \{ 1,2\}, \{6 \} \right),\\
e_4 &= \left( \{ 4,6\}, \{3, 5 \} \right), & e_5 &= \left( \{ 4\}, \{3 \} \right).
\end{align*}
The incidence digraph $\Dc_{\cH}$ is represented in Figure~\ref{B_figure_example1} and its adjacency matrix is given by
\[
A_{{\Dc_{\cH}}} =
\left[
\begin{array}{c|c}
0_{6 \times 6} & W \\
\hline
T & 0_{5 \times 5}
\end{array}
\right]
\]
with 
\[
W = 
\left[
\begin{array}{ccccc}
1 & 0& 0 & 0 & 0  \\
0 & 1& 0 & 0 & 0  \\
0 & 0& 0 & 1 & 1  \\
0 & 1& 0 & 0 & 0  \\
0 & 0& 0 & 1 & 0  \\
0 & 0& 1 & 0 & 0  
\end{array}
\right]
\qquad
\mbox{ and }
\qquad
T = 
\left[
\begin{array}{cccccc}
0 & 1 & 0 & 0 & 1 & 0 \\
0 & 1 & 0 & 0 & 0 & 0 \\
1 & 1 & 0 & 0 & 0 & 0 \\
0 & 0 & 0 & 1 & 0 & 1 \\
0 & 0 & 0 & 1 & 0 & 0 
\end{array}
\right]\, .
\]
\hfill $\Diamond $
\end{exam}

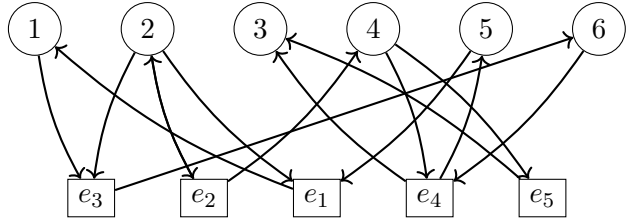
\begin{figure}
\begin{center}
\begin{tikzpicture}
 [scale=.15,auto=left, node distance=1.5cm, 
 ]
 \node[fill=white,style={circle,draw}] (n1) at (4,15) {\small{1}};
 \node[fill=white,style={circle,draw}] (n2) at (14,15) {\small{2}};
 \node[fill=white,style={circle,draw}] (n3) at (24,15)  {\small{3}};
 \node[fill=white,style={circle,draw}] (n4) at (34,15)  {\small{4}};
 \node[fill=white,style={circle,draw}] (n5) at (44,15) {\small{5}};
 \node[fill=white,style={circle,draw}] (n6) at (54,15) {\small{6}};
 \node[fill=white,style={rectangle,draw}] (e3) at (9,0) {\small{$e_3$}};
 \node[fill=white,style={rectangle,draw}] (e2) at (19,0) {\small{$e_2$}};
 \node[fill=white,style={rectangle,draw}] (e1) at (29,0) {\small{$e_1$}};
\node[fill=white,style={rectangle,draw}] (e4) at (39,0) {\small{$e_4$}};
\node[fill=white,style={rectangle,draw}] (e5) at (49,0) {\small{$e_5$}};

\path
       (n2) [->]  edge[bend right=10,thick] node  {} (e1)
        (n5) [->]  edge[bend left=10,thick] node  {}  (e1)
       (n2) [->]  edge[bend right=10,thick] node  {} (e2)
        (n1) [->]  edge[bend right=10,thick] node  {}  (e3)
       (n2) [->]  edge[bend right=10,thick] node  {} (e3)
        (n6) [->]  edge[bend left=10,thick] node  {}  (e4)
       (n4) [->]  edge[bend left=10,thick] node  {}  (e4)
       (n4) [->]  edge[bend left=10,thick] node  {}  (e5)
        (e1) [->]  edge[bend left=10,thick] node {}  (n1)
        (e2) [->]  edge[bend left=10,thick] node {}  (n2)
        (e2) [->]  edge[bend right=10,thick] node {}  (n4)
        (e3) [->]  edge[thick] node {}  (n6)
        (e4) [->]  edge[bend right=10,thick] node {}  (n5)
        (e4) [->]  edge[bend left=10,thick] node {}  (n3)
        (e5) [->]  edge[bend right=10,thick] node {} (n3)
        ;
 \end{tikzpicture} 
\end{center}
\caption{The incidence digraph~$\Dc_{\cH}$ associated with the directed hypergraph in Figure~\ref{figure_example1} (left).}\label{B_figure_example1}
\end{figure}

\newcommand{\FS}{\mathrm{FS}}
\newcommand{\BS}{\textrm{BS}}
\newcommand{\Br}{\textrm{B}}

The \emph{forward star} and the \emph{backward star} of a node~$v$ are the sets of hyperedges defined by $\FS(v) = \set{e}{v \in T(e)}$ and $\BS(v) = \set{e}{v \in H(e)}$, respectively.

\begin{rem}
Note that, in network theory,  the \emph{input set} of a node  in a directed network corresponds  to the backward star of the node. 
\hfill $\Diamond$
\end{rem}

We can  define paths and connectivity in hypergraphs. A \emph{directed path} of lenght $q$ between the nodes~$v_1$ and~$v_{q+1}$ is a sequence of nodes, $v_1, v_2, \dotsc, v_{q+1}$, and directed hyperedges, $e_1, e_2, \dotsc, e_q$, where 
\[
v_1 \in T(e_1),\  v_{q+1} \in {\mathcal H}(e_q), \mbox{ and }  v_j \in {\mathcal H}\left(e_{j-1}\right) \cap T\left(e_{j}\right)  \mbox{ for } j=2, \dotsc, q\, .
\]
The nodes $v_1$ and $v_{q+1}$ are said to be \emph{connected}.   An hypergraph is \emph{(weakly) connected} if every pair of nodes in the hypergraph is connected by a path replacing all of its directed hyperedges with undirected hyperedges.

In the following we assume that all hypergraphs have nonempty node and hyperedge sets and are connected.

\section{Coupled cell hypernetwork formalism} \label{Sec:hypernet}

Weighted directed hypergraphs provide the backbone for the coupled cell hypernetwork formalism that we develop in this work. A \emph{hypernetwork} is a weighted directed hypergraph, where each node $i\in C$ comes with a phase space~$V = \R^{d(i)}$ and internal dynamics~$f_i: V \to V$---we refer to a node with a phase space and internal dynamics as a \emph{cell}.
For simplicity, we assume that all cells are identical, i.e., $V = \R^{d}$ and $f_i=f$ for all~$i$.
Thus, we will use the same symbol for each node/cell in a graphical representation of the network.
In slight abuse of notation and terminology, we write $(\cH, W)$ for the hypernetwork, i.e., the weighted, directed hypergraph together with the data on the phase space, and use the words node/cell interchangeably.

\newcommand{\abs}[1]{\left|#1\right|}

\subsection{Coupled cell hypernetworks}
\label{sec:Hypernetworks}

Fix a hypergraph~$\mathcal{H}=(C, E)$ with nodes~$C$ and hyperedges~$E$; 
in the following all hypergraphs have the same set of nodes~$C$. 
Recall that the backward star of a cell~$c$ is denoted by~$\BS(c)$. For cell~$c$ let
\[
\BS_k(c) = \set{ e \in \BS(c)}{\# T (e) = k} 
\]
denote the set of hyperedges whose tail has cardinality~$k$ and let
\[
\Br(c) = \set{ k}{\exists\ e \in \BS(c) \mbox{ such that } \# T (e) = k} = \set{ k}{\BS_k(c) \ne \emptyset}
\]
be the possible cardinalities. This yields a partition of the backward star since 
\[
\bigcup_{k\in\Br(c)}\BS_k(c) = \BS(c).
\]
Finally, write
\begin{equation}
\label{eq:Orders}
\Br(\mathcal{H}) = \bigcup_{c\in C}\Br(c) = \set{k}{\exists\ e \in E \mbox{ such that } \# T (e) = k}.
\end{equation}

\begin{exam}
Recall the hypergraph~$\Hc$ on the right of Figure~\ref{figure_example1}. We have that 
\begin{align*}
\Br(1) &=  \{ 2\}, & \Br(2) &=  \{ 1\}, &   \Br(3) &=  \{ 1, 2\}  &   &\mbox{ and }  &\Br(\Hc) &=  \{ 1, 2\}.
\end{align*}
\hfill $\Diamond $
\end{exam}

We will now define a set of dynamical equations that is compatible with the hypergraph~$\cH$. For an hyperedge $e\in E$ with weight~$w_e$ we let~$k$ denote the cardinality $\# T(e)$.
The evolution of cell $i\in H(e)$ 
will be determined by a smooth \emph{coupling function~$Q_k:V^{k+1} \to V$} such that the 
evolution of cell~$i$ depends on~$x_i$ and on the~$k$ variables~$x_j$ with $j\in T(e)$.
More precisely, for a hyperedge~$e$ with tail~$T(e)$ of cardinality $\# T (e) = k$, let~$x_{T(e)}$ denote the~$k$ variables in the tail and write $Q_{k}(x_i; x_{T(e)})$.
We assume that~$Q_k$ is invariant under permutation in the last~$k$ variables, the entries of~$x_{T(e)}$.
Note that this implies that each hyperedge~$e'$ with $\#T(e')=k$ is of the same type: The interactions are governed by the same coupling function. At the same time, the strength of the interaction may be different since~$w_e$ may be different from~$w_{e'}$.

\begin{Def}[Admissibility]
A family $Q = (Q_k)$, $k\in\N$, of coupling functions as above is \emph{admissible} for the hypernetwork $(\cH, W)$ if $Q_k \neq 0$ for $k\in\Br(\Hc)$ and $Q_k = 0$ otherwise.
The collection of admissible family of coupling functions~$Q$ define the \emph{admissible cell vector fields}
\begin{align}\label{eq:CCHN}
F_i(x) &= f(x_i) +
\sum_{k\in\Br(i)}
\sum_{e \in \BS_k(i)}
w_{e} Q_{k}\!\left( x_i; x_{T(e)} \right)
\end{align}
for $i \in C$.
\hfill $\Diamond$
\end{Def}

\begin{Def}
Every admissible family of coupling functions~$Q$ for the hypernetwork~$(\cH, W)$ and corresponding cell vector fields~$F_i$ defines a \emph{coupled cell system} where the state~$x_i$ of cell $i\in C$ evolves according to
\[\dot x_i = F_i(x).\]
For convenience, we typically identify the dynamical system and the cell vector fields that define it.
\hfill $\Diamond$
\end{Def}

\begin{exam}
Consider the hypergraph~$\cH$ on the right of Figure~\ref{figure_example1}. For a collection of admissible family of coupling functions~$Q_1,Q_2$, we have that the admissible cell vector fields are given by 
\begin{align*}
\dot{x}_1 &= f(x_1) + Q_2(x_1;x_1,x_2)\\
\dot{x}_2 &= f(x_2) + Q_1(x_2;x_2) \\
\dot{x}_3 &= f(x_3) + Q_1(x_3;x_2) + Q_2(x_3;x_1,x_2)\, .
\end{align*}
\hfill $\Diamond $
\end{exam}

From this perspective, a coupled cell hypernetwork characterizes a set of admissible coupling functions and admissible vector fields. 
However, distinct hypernetworks can have the same set of admissible coupling functions 
{and even the same set of admissible} vector fields.

\begin{exam}\label{example_eq1}
Consider the hypernetwork defined by the hypergraph on the left of Figure~\ref{figure_dif_example1}. For a collection of admissible family of coupling functions~$Q_1,Q_2$,  we  have that the admissible cell vector fields are given by 
\begin{align*}
\dot{x}_1 &= f(x_1) + Q_2(x_1;x_1,x_2)\\
\dot{x}_2 &= f(x_2) + Q_1(x_2;x_3) \\
\dot{x}_3 &= f(x_3) + Q_1(x_3;x_2) + Q_2(x_3;x_1,x_2)\, .
\end{align*}
Note that the directed hypergraph on the right of Figure~\ref{figure_example1} and the one on the left of Figure~\ref{figure_dif_example1} are distinct. Nevertheless, they have the same set of admissible functions, although they do not have the same set of admissible vector fields. 
\hfill $\Diamond $
\end{exam}

\begin{exam}\label{example_eq2}
Consider the hypernetwork defined by the hypergraph on the right of Figure~\ref{figure_dif_example1}. For a collection of admissible family of coupling functions~$Q_1,Q_2$,  
we  have that the admissible cell vector fields are given by 
\begin{align*}
\dot{x}_1 &= f(x_1) + Q_1(x_1;x_1) + Q_2(x_1;x_1,x_2)\\
\dot{x}_2 &= f(x_2) + Q_1(x_2;x_2) + Q_1(x_2;x_3) \\
\dot{x}_3 &= f(x_3) + Q_1(x_3;x_3) + Q_1(x_3;x_2) + Q_2(x_3;x_1,x_2)\, .
\end{align*}
Observe that the two distinct directed hypernetworks of Figure~\ref{figure_dif_example1} have the same set of admissible coupling functions and vector fields. 
\hfill $\Diamond $
\end{exam}

\begin{figure}
\begin{center}
\begin{tabular}{cc}
 \begin{tikzpicture}
 [scale=.15,auto=left, node distance=1.5cm, 
 ]
 \node[fill=pink,style={circle,draw}] (n1) at (4,6)  {\small{1}};
 \node[fill=blue!70,style={circle,draw}] (n2) at (4,-6) {\small{2}};
  \node[fill=white,style={circle,draw}] (n3) at (20,-4)  {\small{3}};
     \node [fill=black,style={circle,scale=0.1}] (e3) at (14,0) { };
        \node [fill=black,style={circle,scale=0.1}] (e33) at (15,0) { };
\path
        (n1) [-]  edge[bend right=20,thick] node { } (e3)
         (n2) [-]  edge[bend left=20,thick] node { } (e3)
         (e3) [->] edge[thick] node {} (e33)
        (e33)  [-]  edge[bend right=90,thick] node {} (n1)  
         (n3) [->] edge[bend left=20, thick]  node {}   (n2) 
         (n2) [->] edge[bend left=20, thick]  node {}   (n3) 
       (e33)  [-]  edge[thick] node {} (n3);    
       
 \end{tikzpicture} &
 \begin{tikzpicture}
 [scale=.15,auto=left, node distance=1.5cm, 
 ]
 \node[fill=pink,style={circle,draw}] (n1) at (4,6)  {\small{1}};
 \node[fill=blue!70,style={circle,draw}] (n2) at (4,-6) {\small{2}};
  \node[fill=white,style={circle,draw}] (n3) at (20,-4)  {\small{3}};
     \node [fill=black,style={circle,scale=0.1}] (e3) at (14,0) { };
        \node [fill=black,style={circle,scale=0.1}] (e33) at (15,0) { };
\path
       (n1)  [->]  edge[loop left=90,thick] node {} (n1)
       (n2)  [->]  edge[loop left=90,thick] node {} (n2)
       (n3)  [->]  edge[loop right=90,thick] node {} (n3)
        (n1) [-]  edge[bend right=20,thick] node { } (e3)
         (n2) [-]  edge[bend left=20,thick] node { } (e3)
         (e3) [->] edge[thick] node {} (e33)
        (e33)  [-]  edge[bend right=90,thick] node {} (n1)  
         (n3) [->] edge[bend left=20, thick]  node {}   (n2) 
         (n2) [->] edge[bend left=20, thick]  node {}   (n3) 
       (e33)  [-]  edge[thick] node {} (n3);    
       
 \end{tikzpicture}
 \end{tabular}
 \end{center}
 \caption{Two distinct directed hypernetworks with the same admissible vector fields. Assume all hyperedges have weight $1$.}\label{figure_dif_example1}
\end{figure}
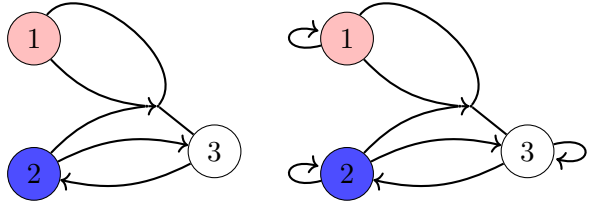

\begin{exam}
Consider the hypernetworks $(\Hc,W_1)$ (left) and $(\Hc,W_2)$ (right) of Figure~\ref{fig_weighted_dif_same_hyper}. 
Thus, the same hypergraph~$\Hc$ and different weighted adjacency matrices and, thus, different admissible vector fields. In fact, for an admissible coupling function~$Q_1$,  
we have that the admissible cell vector fields for $(\Hc,W_1)$ are given by 
\begin{align*}
\dot{x}_1 &= f(x_1), \\
\dot{x}_2 &= f(x_2) + Q_1(x_2;x_1) + Q_1(x_2;x_3), \\
\dot{x}_3 &= f(x_3) + Q_1(x_3;x_1) + Q_1(x_3;x_2);
\end{align*}
and the admissible cell vector fields for $(\Hc,W_2)$ are given by 
\begin{align*}
\dot{x}_1 &= f(x_1), \\
\dot{x}_2 &= f(x_2) + Q_1(x_2;x_1) + Q_1(x_2;x_3), \\
\dot{x}_3 &= f(x_3) + Q_1(x_3;x_1) + 3Q_1(x_3;x_2);
\end{align*}
Thus,  we see that $(\Hc,W_1)$ and  $(\Hc,W_2)$  have distinct set of admissible  vector fields. 
\hfill $\Diamond $
\end{exam}

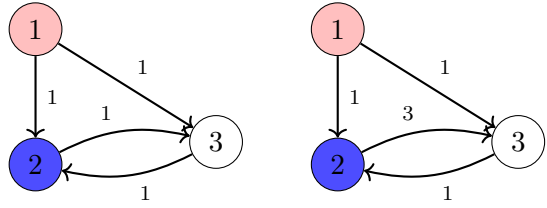
\begin{figure}
\begin{center}
\begin{tabular}{cc}
 \begin{tikzpicture}
 [scale=.15,auto=left, node distance=1.5cm, 
 ]
 \node[fill=pink,style={circle,draw}] (n1) at (4,6)  {\small{1}};
 \node[fill=blue!70,style={circle,draw}] (n2) at (4,-6) {\small{2}};
  \node[fill=white,style={circle,draw}] (n3) at (20,-4)  {\small{3}};
\path
        (n1) [->]  edge[bend right=0,thick] node {{\tiny $1$}} (n3)
         (n1) [->]  edge[bend left=0,thick] node {{\tiny $1$}} (n2)
         (n3) [->] edge[bend left=20, thick]  node {{\tiny $1$}} (n2) 
         (n2) [->] edge[bend left=20, thick]  node {{\tiny $1$}} (n3);         
 \end{tikzpicture} \quad & \quad 
 \begin{tikzpicture}
 [scale=.15,auto=left, node distance=1.5cm, 
 ]
 \node[fill=pink,style={circle,draw}] (n1) at (4,6)  {\small{1}};
 \node[fill=blue!70,style={circle,draw}] (n2) at (4,-6) {\small{2}};
  \node[fill=white,style={circle,draw}] (n3) at (20,-4)  {\small{3}};
\path
        (n1) [->]  edge[bend right=0,thick] node {{\tiny $1$}} (n3)
         (n1) [->]  edge[bend left=0,thick] node {{\tiny $1$}} (n2)
         (n3) [->] edge[bend left=20, thick]  node {{\tiny $1$}} (n2) 
         (n2) [->] edge[bend left=20, thick]  node {{\tiny $3$}} (n3);         
 \end{tikzpicture} 
 \end{tabular}
 \end{center}
 \caption{A directed hypergraph with different weighted adjacency matrices (and different admissible vector fields). }
 \label{fig_weighted_dif_same_hyper}
\end{figure}

\begin{Def}\label{def:HypernetworkEquiv}
Two hypernetworks $(\Hc_1,W_1)$ and $(\Hc_2,W_2)$ with identical cells (i.e., the nodes, their phase space, and internal dynamics) are \emph{identical as coupled cell systems} if they have the same set of admissible cell vector fields.
Two hypernetworks $(\Hc_1,W_1)$ and $(\Hc_2,W_2)$ with identical cells are \emph{equivalent as coupled cell systems} if they are identical up to a permutation of the cells.
\hfill $\Diamond$
\end{Def}

\begin{exam}
The two directed, weighted hypergraphs in Figure~\ref{figure_dif_example1} are identical {(and equivalent)}   
as coupled cell hypernetworks as outlined in Examples~\ref{example_eq1} and~\ref{example_eq2}. 
\hfill $\Diamond$
\end{exam}

\begin{exam}\label{exam_equiv_lemma}
The two hypernetworks defined by the hypergraphs in Figure~\ref{fig_equiv_lemma} are identical (equivalent) as coupled cell hypernetworks. 
For an admissible coupling function~$Q_2$, we have that for both coupled cell hypernetworks, the admissible cell vector fields are given by 
\begin{align*}
\dot{x}_1 &= f(x_1) \\
\dot{x}_2 &= f(x_2) \\
\dot{x}_3 &= f(x_3) + Q_2(x_3;x_1, x_2)\\
\dot{x}_4 &= f(x_4) + Q_2(x_3;x_1,x_2)
\end{align*}
\hfill $\Diamond$
\end{exam}

\begin{figure}
\begin{center}
\begin{tabular}{cc}
 \begin{tikzpicture}
 [scale=.15,auto=left, node distance=1.5cm, 
 ]
 \node[fill=white,style={circle,draw}] (n1) at (4,6)  {\small{1}};
 \node[fill=white,style={circle,draw}] (n2) at (4,-6) {\small{2}};
  \node[fill=white,style={circle,draw}] (n3) at (24,6)  {\small{3}};
  \node[fill=white,style={circle,draw}] (n4) at (24,-6)  {\small{4}};
     \node [fill=black,style={circle,scale=0.1}] (e3) at (14,0) { };
        \node [fill=black,style={circle,scale=0.1}] (e33) at (15,0) { };
\path
        (n1) [-]  edge[bend right=20,thick] node { } (e3)
         (n2) [-]  edge[bend left=20,thick] node { } (e3)
         (e3) [->] edge[thick] node {} (e33)
        (e33)  [-]  edge[bend right=20,thick] node {} (n3)  
         (e33)  [-]  edge[bend left=20,thick] node {} (n4);    
 \end{tikzpicture} \quad & \quad 
 \begin{tikzpicture}
 [scale=.15,auto=left, node distance=1.5cm, 
 ]
 \node[fill=white,style={circle,draw}] (n1) at (4,6)  {\small{1}};
 \node[fill=white,style={circle,draw}] (n2) at (4,-6) {\small{2}};
  \node[fill=white,style={circle,draw}] (n3) at (24,6)  {\small{3}};
  \node[fill=white,style={circle,draw}] (n4) at (24,-6)  {\small{4}};
     \node [fill=black,style={circle,scale=0.1}] (e3) at (14,3) { };
     \node [fill=black,style={circle,scale=0.1}] (e33) at (15,3) { };
     \node [fill=black,style={circle,scale=0.1}] (e4) at (14,-3) { };
     \node [fill=black,style={circle,scale=0.1}] (e44) at (15,-3) { };       
\path
        (n1) [-]  edge[bend right=20,thick] node { } (e3)
         (n2) [-]  edge[bend left=20,thick] node { } (e3)
         (e3) [->] edge[thick] node {} (e33)
        (e33)  [-]  edge[bend right=20,thick] node {} (n3)  
        
        (n1) [-]  edge[bend right=20,thick] node { } (e4)
         (n2) [-]  edge[bend left=20,thick] node { } (e4)
         (e4) [->] edge[thick] node {} (e44)
        (e44)  [-]  edge[bend right=20,thick] node {} (n4);  
 \end{tikzpicture}
 \end{tabular}
 \end{center}
 \caption{Two identical (equivalent) coupled cell hypernetworks corresponding to two distinct weighted directed hypergraphs. 
 Here, we are assuming all hyperedges with weight $1$.}\label{fig_equiv_lemma}
\end{figure}
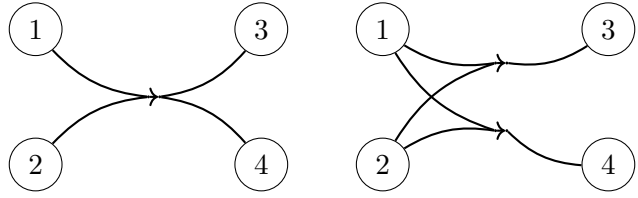

\begin{exam}
The two hypernetworks in Figure~\ref{fig_one_is_cont_other} are not equivalent as coupled cell hypernetworks. For an admissible coupling function~$Q_2$,  
we have that the admissible cell vector fields for the hypergraph on the left are given by 
\begin{align*}
\dot{x}_1 &= f(x_1) \\
\dot{x}_2 &= f(x_2)\\
\dot{x}_3 &= f(x_3) + Q_2(x_3;x_1,x_2)
\end{align*}
where~$Q_2$ is invariant under permutation of the last two coordinates. 
For an admissible coupling function~$Q_1$,  
we have that the admissible cell vector fields for the hypergraph on the right are given by 
\begin{align*}
\dot{x}_1 &= f(x_1) \\
\dot{x}_2 &= f(x_2) \\
\dot{x}_3 &= f(x_3) + Q_1(x_3;x_1) + Q_1(x_3;x_2)\, .
\end{align*}
Note that the function $Q_1(x_3;x_1) + Q_1(x_3;x_2)$ is a particular case of $Q_2(x_3;x_1,x_2)$. 
That is, fixing the same cell phase spaces for the two hypergraphs, we have that the set of admissible cell vector fields for the hypergraph on the right is strictly contained  in the set of admissible cell vector fields for the hypergraph on the left.
\hfill $\Diamond$
\end{exam}

\begin{figure}
\begin{center}
\begin{tabular}{cc}
 \begin{tikzpicture}
 [scale=.15,auto=left, node distance=1.5cm, 
 ]
 \node[fill=white,style={circle,draw}] (n1) at (4,6)  {\small{1}};
 \node[fill=white,style={circle,draw}] (n2) at (4,-6) {\small{2}};
  \node[fill=white,style={circle,draw}] (n3) at (20,0)  {\small{3}};
     \node [fill=black,style={circle,scale=0.1}] (e3) at (14,0) { };
        \node [fill=black,style={circle,scale=0.1}] (e33) at (15,0) { };
\path
        (n1) [-]  edge[bend right=20,thick] node { } (e3)
         (n2) [-]  edge[bend left=20,thick] node { } (e3)
         (e3) [->] edge[thick] node {} (e33)
        (e33)  [-]  edge[bend right=0,thick] node {} (n3);    
 \end{tikzpicture} \quad & \quad 
 \begin{tikzpicture}
 [scale=.15,auto=left, node distance=1.5cm, 
 ]
 \node[fill=white,style={circle,draw}] (n1) at (4,6)  {\small{1}};
 \node[fill=white,style={circle,draw}] (n2) at (4,-6) {\small{2}};
  \node[fill=white,style={circle,draw}] (n3) at (20,0)  {\small{3}};
\path
        (n1) [->]  edge[bend right=0,thick] node { } (n3)
         (n2) [->]  edge[bend left=0,thick] node { } (n3);
 \end{tikzpicture} 
 \end{tabular}
 \end{center}
 \caption{Two distinct directed hypernetworks. Assume all hyperedges have weight $1$. For any choice of cell phase spaces, the set of admissible vector fields for the hypernetwork on the right is strictly contained at the set of admissible vector fields for the hypernetwork on the left.}\label{fig_one_is_cont_other}
\end{figure}
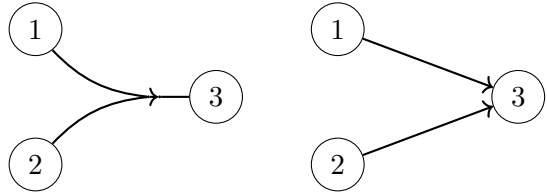

\begin{lemma}\label{lem:AdmEquiv}
A weighted directed hypergraph~$(\Hc,W)$ is equivalent as a coupled cell hypernetwork to a  weighted directed hypergraph~$(\Hc',W')$ such that the head~$H(e)$ of any hyperedge~$e\in E(\Hc')$ has cardinality 1.
\end{lemma}

\begin{proof}
Replace any {hyperedge} $e\in E(\Hc)$ with head set $H(e) = \left\{v_1, \dotsc, v_k\right\}$ where $k>1$, and weight~$w_e$ by~$k$ {hyperedges} $e_j= \left(T(e) , \{v_j\} \right)$, for $j=1, \dotsc, k$, each with weight $w_{e_j}=w_e$. 
The set of admissible coupling functions and vector fields remain unchanged since they only depend on the tail of any {hyperedge}.
\end{proof}

\subsection{{Hyperedge}-Maximality, {hyperedge}-minimality, and symmetries}

In the previous section, we characterized a hypernetwork based on {its} set of admissible coupling functions/vector fields. In this section, we will now change perspective and focus on a specific choice of coupling function. Indeed, for a specific choice of coupling functions, we obtain a specific vector field.

\begin{Def}
A hypernetwork~$(\Hc,W)$ and an admissible family of coupling functions $Q = (Q_1, Q_2, \dotsc)$ defines a \emph{hypernetwork coupling} $(\Hc, W, Q)$ with associated cell vector field~$F$ as in~\eqref{eq:CCHN}.
\hfill $\Diamond$
\end{Def}

Conversely, we can assign a hypernetwork coupling to a dynamical system. 

\begin{Def}
A network dynamical system determined by $x_i\in V$, $i\in C$, evolving according to
\begin{equation}
\dot x_i = X_i(x)
\end{equation}
is a \emph{coupled cell system for a hypernetwork coupling $(\Hc, W, Q)$} if $X_i=F_i$ for an admissible cell vector field~$F_i$ with respect to $(\Hc, W, Q)$ as defined in~\eqref{eq:CCHN}.
\hfill $\Diamond$
\end{Def}

Note that the assignment of a hypernetwork coupling to a dynamical system is not unique since the hypergraph and coupling function go hand in hand. Lemma~\ref{lem:AdmEquiv} already indicated that even on the level of admissible vector fields, there are different hypergraphs that give rise to the same set of admissible coupling functions/vector fields. See Example~\ref{exam_equiv_lemma} and Figure~\ref{fig_equiv_lemma}.

\begin{Def}
Two hypernetwork couplings $(\Hc, W, Q)$, $(\Hc', W', Q')$ are \emph{identical} if 
the induced coupled cell system is the same, that is, the corresponding cell vector fields~$F,F'$ satisfy $F = F'$. 
Two hypernetwork couplings $(\Hc, W, Q)$, $(\Hc', W', Q')$ are \emph{equivalent} if they are identical up to a permutation of the cells.
\hfill $\Diamond$
\end{Def}

\begin{exam}\label{exam:Split}
Consider the hypernetwork couplings $(\Hc, W, Q)$ with
\begin{align*}
E(\Hc) &= \sset{(\sset{1,\dotsc,N}, \sset{1,\dotsc,N})},\\
W &= (1),\\
Q_N(x_i; x_1, \dotsc, x_N) & = \prod_{j=1}^N x_j + \sum_{j=1}^N x_j, 
\intertext{and $(\Hc', W', Q')$ with} 
E(\Hc') &=E(\Hc) \cup\sset{(\sset{1}, \sset{1,\dotsc,N}), \dotsc, (\sset{N}, \sset{1,\dotsc,N})},\\
W' &= (1, 1, \dotsc, 1),\\
Q'_N(x_i; x_1, \dotsc, x_N) &= \prod_{j=1}^N x_j  \text{ and } {Q'_1 (x_i;x_1) = x_1}.
\end{align*}
These hypernetwork couplings are identical.
\hfill $\Diamond$
\end{exam}

This implies that we can get equivalent hypernetwork couplings by splitting, or conversely combining {hyperedges}.

\begin{Def}
Suppose that~$(\Hc, W, Q)$ is a hypernetwork coupling and let $e\in{E(\Hc)}$ be an {hyperedge}. The hypernetwork coupling $(\Hc', W', Q')$ arises by \emph{splitting the {hyperedge}~$e$ into {hyperedges} $e'_1, \dotsc, e'_k$} if $(\Hc, W, Q)$ and $(\Hc', W', Q')$ are identical and $E(\Hc') = (E(\Hc)\smallsetminus \sset{e})\cup\sset{e'_1, \dotsc, e'_k}$. Conversely, $(\Hc, W, Q)$ arises from $(\Hc', W', Q')$ by \emph{combining the {hyperedges} $e'_1, \dotsc, e'_k$}.
\hfill $\Diamond$
\end{Def}

The hypernetwork couplings in Example~\ref{exam:Split} can be obtained by splitting/combining {hyperedges}.

Note that we do not require~$e$ to be distinct from $e'_1, \dotsc, e'_k$, we do not require $\sset{e'_1, \dotsc, e'_k}$ to be disjoint from~$E(\Hc)$, nor do we necessarily have $Q\neq Q'$. If $Q=Q'$ then the splitting/combining an {hyperedge} is \emph{purely structural}.

\begin{Def} \label{Def:operations} 
Given an hypernetwork $(\Hc, W)$ we define the following \emph{purely structural hyperedge operations}:
\begin{enumerate}
\item Any hyperedge ${e}\in E(\Hc)$ with weight~$w_e$ and 
head 
 ${H(e)} = \sset{v_1, \dotsc, v_r}$ can be split into~${r}$ {hyperedges} $e_l = (t, \sset{v_l})$, $l=1, \dotsc, r$ each with weight~$w_e$;
\item {More generally, 
any hyperedge $e \in E(\Hc)$ with weight~$w_e$ and 
head
$H(e) = H_1 \cup \dotsc \cup H_r$, with $H_i \ne\emptyset$ and $H_i \cap H_j=\emptyset$, for $i \ne j$, can be split into~$r$ {hyperedges} $e_l = (T, H_l)$, $l=1, \dotsc, r$ each with weight~$w_e$;
}
\item Conversely, two hyperedges $e_1, e_2$ with $T(e_1)=T(e_2)$ can be combined into a single {hyperedge} if their heads are disjoint, $H(e_1)\cap H(e_2)=\emptyset$, and they have the same weight.
\end{enumerate}
\hfill $\Diamond$
\end{Def}

The following property is immediate:

\begin{lemma} \label{lem:ident_split}
Let $(\Hc, W)$ and $(\Hc', W')$ be two hypernetworks such that $(\Hc', W')$ is obtained from $(\Hc, W)$ by one (or more) purely structural splitting/combining {hyperedge} operations.
Then the hypernetworks are identical as coupled cell systems.
Moreover, for every family of admissible coupling functions~$Q$, the hypergraph couplings $(\Hc,W,Q)$ and $(\Hc', W',Q)$ are identical.
\end{lemma}

Splitting an hyperedge does not necessarily increase  
the number of hyperedges. Indeed, if $\sset{e'_1, \dotsc, e'_k}\subset E(\Hc)$ then the {hyperedge}~$e$ is \emph{redundant} and splitting the {hyperedge} decreases the overall number of {hyperedges}.

\begin{exam}\label{exam:Redundant}
Let $e=\left( \{1,\dotsc,N\}, \{1,\dotsc,N\}\right)$. Consider $(\Hc, W, Q)$ with 
\[
E(\Hc)
 = \left\{ e, \left( \sset{1}, \sset{1,\dotsc,N}\right), \dotsc, \left( \sset{N}, \sset{1,\dotsc,N}\right) \right\} 
 \] and 
$Q_N \left( x_i; x_1, \dotsc, x_N\right) = \sum_{j=1}^Nx_j$ and $Q_1 (x_i; x_1)= x_1$. Then the hyperedge~$e$ is redundant. Note that redundancy here depends on the specific form of the coupling functions.
\hfill $\Diamond$
\end{exam}

To any arbitrary hypernetwork coupling we can associate a maximal and minimal dynamically equivalent hypernetwork coupling.

\begin{Def}
A hypernetwork coupling~$(\Hc, W, Q)$ is \emph{hyperedge-maximal} if  
no {hyperedge} can be split {to obtain an equivalent hypergraph coupling}. 
Conversely, a hypernetwork coupling is \emph{hyperedge-minimal} if  
no {hyperedges} can be joined to obtain an equivalent hypergraph coupling structure. 
\hfill $\Diamond$
\end{Def}

Note that without further assumptions, neither hyperedge-minimal nor -maximal associated hypernetwork couplings need to be unique: For example, if a hypernetwork coupling has an associated minimal hypernetwork coupling that has a single hyperedge~$e$ with weight~$w_e$ then we get an infinite family of minimal hypernetwork couplings for $w'_e = aw_e$ and $Q_e'=a^{-1}Q_e$, $a\in\R \setminus \{0\}$.

\begin{Def}\label{def:Proper}
A hypernetwork coupling~$(\cH, W, Q)$ is \emph{proper} if all its associated {hyperedge-maximal} hypernetwork couplings contain at least one hyperedge that is not an edge of a graph, i.e., an edge that is not of the form $e=(\sset{t}, \sset{h})$ with $t,h \in C(\cH)$.
\hfill $\Diamond$
\end{Def}

\begin{exam}
The coupled cell system defined in Example~\ref{exam:Redundant} is not proper: An associated hyperedge-maximal hypernetwork coupling has edges 
\[
E(\Hc)=
\left\{ \left( \{1\},\{1\}\right), \left(\{1\},\{2\}\right), \dotsc, \left(\{N\},\{N-1\}\right), \left(\{N\},\{N\}\right)\right\}
\]
and 
\noindent
$Q_1\left( x_i; x_1\right) = x_1$.
 However, if~$Q_N$ is substituted with 
 {$Q_N'$ defined by 
 $Q_N' \left( x_i; x_1, \dotsc, x_N\right) = x_1\dotsb x_N$} 
 then any associated maximal coupling structure must have $\left(\sset{1, \dotsc, N}, \{i\}\right) \in E(\Hc)$, $i=1,\dotsc,N$ and thus yields a proper coupled cell hypernetwork.
\hfill $\Diamond$
\end{exam}

Note that we can always split {hyperedges} whose heads have cardinality greater than one. The following is an immediate consequence of Lemma~\ref{lem:AdmEquiv}:

\begin{lemma}
Consider a coupled cell system with associated maximal hypernetwork coupling~$(\Hc, W, Q)$. If $(t,h)\in E(\Hc)$ then $\#{h}=1$.
\end{lemma}

We now explore some straightforward consequences of equivalent hypernetwork couplings and how they relate to the symmetry of the coupled cell hypernetworks they define. Let~$\mathbf{S}_N$ denote the symmetric group of~$N$ elements that acts by permuting the node indices.

\begin{prop}\label{prop:FullSym}
Write $e=(\sset{1,\dotsc,N}, \sset{1,\dotsc,N})$ and consider a coupled cell system. 
 If an associated hyperedge-minimal hypernetwork coupling~$(\Hc, W, Q)$ has exactly one edge~$e$, i.e., $E(\Hc) = \sset{e}$, 
 then the coupled cell hypernetwork is $\mathbf{S}_N$-equivariant.
\end{prop}

\begin{proof}
The existence of a minimal hypernetwork coupling~$(\Hc, W, Q)$ with $E(\Hc) = \sset{e}$ implies that
\begin{equation}
\dot x_i = f(x_i) + w_{e}Q_{N}(x_i; x_1, \dotsc, x_N) \quad \left(i \in C\right), 
\end{equation}
all cells are globally and identically coupled. These equations are $\mathbf{S}_N$-equivariant.
\end{proof}

More generally we can make the following statement.

\begin{prop}
Consider a coupled cell system. Suppose that there is an associated hypernetwork coupling~$(\Hc, W, Q)$ and a set $A\subset C(\Hc)$ of cells such that for any edge $(t,h)\in E(\Hc)$ we have
(a)~if $a\in h\cap A$ then $A\subset h$ or
(b)~if $a\in t\cap A$ then $A\subset t$.
Then the coupled cell system is $\mathbf{S}_k$-equivariant where $k=\#{A}$ and $\mathbf{S}_k$ acts by permuting the vertices in~$A$.
\end{prop}

\begin{proof}
By definition of a coupled cell system,
Property~(a) ensures that any node in~$A$ receives the same input. At the same time, Property~(b) ensures that the input of any node depends in the same way on all nodes contained in~$A$ consequently, permuting nodes with indices in~$A$ does not affect the dynamical equations which proves $\mathbf{S}_k$-equivariance.
\end{proof}

Of course Proposition~\ref{prop:FullSym} is a special case of the previous statement with $A=C(\Hc)$.

\section{Synchrony in Coupled Cell Hypernetworks} \label{sec:sync}

Synchrony and synchrony patterns---where different nodes in the network evolve identically---is an essential collective phenomenon in network dynamical systems. Given a hypernetwork, what are the possible synchrony patterns for any admissible vector field? 
In the following we describe the synchrony patterns of a coupled cell hypernetwork and their associated balanced relations and quotient hypernetworks.

\subsection{Input sets} As a first step, we generalize the concept of input equivalence relation for networks to the hypernetworks. 
For standard $n$-node directed graphs, Definition~3.2 of~\cite{SGP03} introduces the concept of input equivalence of nodes. Roughly, two nodes~$c$ and~$c'$ are said to be \emph{input equivalent} when besides the number of directed edges to~$c$ and~$c'$ is the same there is also 
a bijection between those sets of directed edges which preserves the edge types.

\begin{Def}\label{def:InputEquiv}
Consider a weighted directed hypernetwork with set of nodes~$C$, set of hyperedges~$E$ 
and weight matrix $W$.
Recall from Section~\ref{sec:Hypernetworks} that~$\Br(c)$ denotes the cardinalities of the hyperedges adjacent to $c\in C$.
Define the \emph{input equivalence relation~$\simI$} on~$C$ in the following way: \\
(i) Cells with empty backward star are input equivalent, as we are assuming all cells are identical. \\
(ii) Two cells $c, c' \in C$ with nonempty backward star are input equivalent if and only if
\begin{enumerate}[label=(\alph*)] 
\item 
$\Br(c) = \Br(c')$;
\item 
For all $k \in \Br(c)$ we have $\sum_{e\in\BS_k(c)} w_{e} = \sum_{e\in\BS_k(c')} w_{e},$
where~$w_e$ denotes the weight of the hyperedge~$e$.
\end{enumerate}
 \hfill $\Diamond $
\end{Def}

In the above definition for two cells to be input equivalent, condition~(iia) imposes that the sets of all cardinalities of the tail sets of the hyperedges of both cells must coincide. Moreover, condition~(iib) says that, for a fixed cardinality of the tail set of a hyperedge of a cell, the summation of the weights of all the edges with the same tail set cardinality must coincide for both cells. 

\begin{exam} (i) Consider the directed hypernetwork in Figure~\ref{h_net_1_v2}. We have that $\simIo = \{ \{1\}, \{2\}, \{3\}, \{4\}\}$. Note that $\BS(1) = \{e_3\}, \BS(2) = \emptyset$, $\BS(3) =\{e_1\}$ and $\BS(4) = \{e_1,e_2\}$ where $\# T (e_1) =2$ and $\# T (e_2) =\# T(e_3) =1$.
\\
(ii) Consider the weighted directed hypernetwork on the right of  Figure~\ref{net1} and the directed network on the left. If $a+b = c+d$, then ${\simIo} = \{ \{1,2\}, \{3, 4\}\}$ for both. \\
(iii) Consider the weighted directed hypernetwork in Figure~\ref{h_net_2} with hyperedges
\begin{align*}
e_1 &= \left( \{1,2\}, \{3\} \right), & e_2 &= \left( \{5,6\}, \{2,3\} \right),\\
e_3 &= \left( \{1,2\}, \{4\} \right), & e_4 &= \left( \{2,3,4\}, \{5\} \right).
\end{align*}
We have that $\BS(3) = \{ e_1, \, e_2\}$, $\BS(4) = \{ e_3\}$ and $w_{e_1} + w_{e_2} = 2 = w_{e_3}$. Thus $3 \simI 4$. In fact, we have 
that $\simIo = \left\{ \{1,6\},\, \{2\},\, \{3,4\},\, \{5\}\right\}$. 
\hfill $\Diamond $
\end{exam}

\begin{figure}
\begin{center}
\begin{tikzpicture}
 [scale=.15,auto=left, node distance=1.5cm, 
 ]
 \node[fill=white,style={circle,draw}] (n1) at (4,0) {\small{1}};
 \node[fill=white,style={circle,draw}] (n2) at (24,0) {\small{2}};
 \node[fill=white,style={circle,draw}] (n3) at (14,9)  {\small{3}};
 \node[fill=white,style={circle,draw}] (n4) at (14,-9)  {\small{4}};
 \node[fill=white,style={circle,draw}] (n5) at (44,0) {\small{5}};
 \node[fill=white,style={circle,draw}] (n6) at (44,9) {\small{6}};
 \node [fill=black,style={circle,scale=0.1}] (e1) at (14,3) { };
  \node [fill=black,style={circle,scale=0.1}] (e11) at (14,5) { };
 \node [fill=black,style={circle,scale=0.1}] (e2) at (14,-3) { };
 \node [fill=black,style={circle,scale=0.1}] (e22) at (14,-5) { };
 \node [fill=black,style={circle,scale=0.1}] (e3) at (38,0) { };
  \node [fill=black,style={circle,scale=0.1}] (e33) at (40,0) { };
 \node [fill=black,style={circle,scale=0.1}] (e4) at (30,6) { };
  \node [fill=black,style={circle,scale=0.1}] (e44) at (28,6) { };
\path
        (n1) [-]  edge[bend right=10,thick] node { } (e1)
        (n2) [-]  edge[bend left=10,thick] node { } (e1)
        (e1) [->] edge[thick] node  {} (e11)
        (e11) [-] edge[thick] node  [near start] {{\tiny $1$}} (n3)
        
        (n1) [-]  edge[bend left=10,thick]  node { } (e2)
        (n2) [-]  edge[bend right=10,thick] node { } (e2)
        (e2) [->] edge[thick] node  {} (e22)    
        (e22) [-] edge[thick] node  [near start]  {{\tiny $2$}} (n4)    

        (n3) [-]  edge[bend right=10,thick]  node { } (e3)
        (n2) [-]  edge[thick] node { } (e3)
         (n4) [-]  edge[bend left=10,thick] node { } (e3)
        (e3) [->] edge[thick] node    {} (e33) 
        (e33) [-] edge[thick] node  [below]  {{\tiny $1$}} (n5) 

       (n5) [-]  edge[bend right=10,thick]  node { } (e4)
        (n6) [-]  edge[bend left=10,thick]  node { } (e4)
        (e4) [->]  edge[thick] node { } (e44)
        (e44) [-] edge[bend left=10,thick]  node  [above]  {} (n3)     
           (e44) [-] edge[bend right=10,thick]  node   [above]  {{\tiny $1$}} (n2)      
        ;
 \end{tikzpicture} 
\end{center}
\caption{A weighted directed graph with six nodes. }\label{h_net_2}
\end{figure}
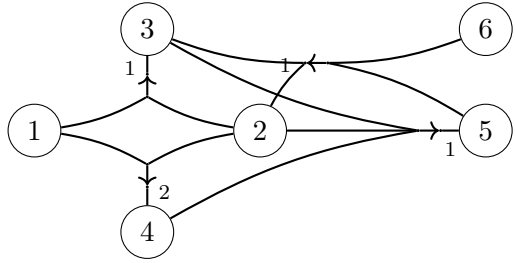

\begin{exam}\label{ex:HypNet}
Consider the directed hypernetwork on the left in Figure~\ref{hyp_ff} with set of nodes $\{1, \dotsc, 6\}$ and five hyperedges, all with weight~$1$: 
\begin{align*}
e_1 &= \left( \{1,2\}, \{4\}\right), &
e_2 &= \left( \{1,2,3\}, \{5\}\right), &
e_3 &= \left( \{4,5\}, \{6\}\right),\\
e_4 &= \left( \{1,2\}, \{1\}\right),& 
e_5 &= \left( \{1,2,3\}, \{2,3\}\right).
\end{align*}
Thus $\simIo = \sset{ \{1,4,6\},\, \{ 2,3,5\} }$.  The admissible equations for this hypernetwork are 
\begin{align*}
\dot{x}_1 &= f(x_1) + Q_2(x_1; x_1,x_2) \\
\dot{x}_2 &= f(x_2) + Q_3(x_2; x_1,x_2,x_3) \\
\dot{x}_3 &= f(x_3) + Q_3(x_3; x_1,x_2,x_3) \\
\dot{x}_4 &= f(x_4) + Q_2(x_4; x_1,x_2) \\
\dot{x}_5 &= f(x_5) + Q_3(x_5; x_1, x_2, x_3)\\
\dot{x}_6 &= f(x_6) + Q_2(x_6; x_4, x_5) 
\end{align*}
where~$Q_2$ and~$Q_3$ are invariant under permutation of the last two and three variables, respectively.

Observe that the set
\[\Delta = \set{x}{x_1 = x_4 = x_6,\ x_2 = x_3 = x_5}\]
is flow-invariant for the above equations and the restriction of those equations to~$\Delta$ is given by
\begin{align*}
\dot{x}_1 &= f(x_1) + Q_2(x_1; x_1,x_2),\\
\dot{x}_2 &= f(x_2) + Q_3(x_2; x_1,x_2,x_2).
\end{align*}
These equations are admissible by the hypernetwork on the right in Figure~\ref{hyp_ff}. 
This motivates the notion of a quotient hypernetwork; we make this explicit in the following section.
\hfill $\Diamond$
\end{exam}

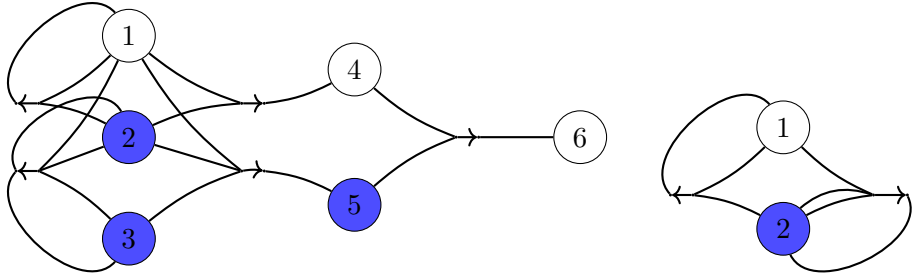
\begin{figure}
\begin{tabular}{cc}
\begin{tikzpicture}
 [scale=.15,auto=left, node distance=1.5cm, 
 ]
 \node[fill=white,style={circle,draw}] (n1) at (4,9)  {\small{1}};
 \node[fill=blue!70,style={circle,draw}] (n2) at (4,0) {\small{2}};
 \node[fill=blue!70,style={circle,draw}] (n3) at (4,-9)  {\small{3}};
 \node[fill=white,style={circle,draw}] (n4) at (24,6) {\small{4}};
   \node[fill=blue!70,style={circle,draw}] (n5) at (24,-6)  {\small{5}};
 \node[fill=white,style={circle,draw}] (n6) at (44,0)  {\small{6}};
 \node [fill=black,style={circle,scale=0.1}] (e1) at (14,3) { };
 \node [fill=black,style={circle,scale=0.1}] (e11) at (16,3) { };
 \node [fill=black,style={circle,scale=0.1}] (e2) at (14,-3) { };
 \node [fill=black,style={circle,scale=0.1}] (e22) at (16,-3) { };
  \node [fill=black,style={circle,scale=0.1}] (e3) at (33,0) { };
   \node [fill=black,style={circle,scale=0.1}] (e33) at (35,0) { };
  \node [fill=black,style={circle,scale=0.1}] (e4) at (-4,3) { };
  \node [fill=black,style={circle,scale=0.1}] (e44) at (-6,3) { };
  \node [fill=black,style={circle,scale=0.1}] (e5) at (-4,-3) { }; 
   \node [fill=black,style={circle,scale=0.1}] (e55) at (-6,-3) { }; 
\path
        (n1) [-]  edge[bend right=10, thick]  node { } (e1)
        (n2) [-]  edge[bend left=10,thick]  node { } (e1)
        (e1) [->] edge[thick]      node  {}  (e11)
        (e11) [-] edge[bend right=10,thick]      node  {}  (n4)

        (n1) [-]  edge[bend right=10,thick]  node { } (e2)
        (n2) [-]  edge[thick] node { } (e2)
        (n3) [-]  edge[bend left=10,thick] node { } (e2)
        (e2) [->] edge[bend left=10,thick]  node {}   (e22)  
        (e22) [-] edge[bend left=10,thick]  node {}   (n5)         
        (n4) [-]  edge[bend right=10,thick]  node { } (e3)
        (n5) [-]  edge[bend left=10,thick]  node { } (e3)
        (e3) [->] edge[thick] node {}   (e33)   
        (e33) [-] edge[thick] node {}   (n6)      
        (n1) [-]  edge[bend left=10,thick] node { } (e4)
        (n2) [-]  edge[bend right=10,thick]  node { } (e4)
        (e4)  [->]  edge[thick] node  {} (e44)
         (e44)  [-]  edge[bend left=90,thick] node [left] {} (n1)
        (n1) [-]  edge[bend left=10,thick]  node { } (e5)
        (n2) [-]  edge[thick] node { } (e5)
         (n3) [-]  edge[bend right=10,thick]  node { } (e5)
          (e5)  [->]  edge[thick] node {} (e55)
        (e55)  [-]  edge[bend right=90,thick] node [left] {} (n3)
        (e55)  [-]  edge[bend left=90,thick] node [left] {} (n2);
        
 \end{tikzpicture}  &
 \begin{tikzpicture}
 [scale=.15,auto=left, node distance=1.5cm, 
 ]
 \node[fill=white,style={circle,draw}] (n1) at (4,9)  {\small{1}};
 \node[fill=blue!70,style={circle,draw}] (n2) at (4,0) {\small{2}};
  \node [fill=black,style={circle,scale=0.1}] (e4) at (-4,3) { };
   \node [fill=black,style={circle,scale=0.1}] (e44) at (-6,3) { };
  \node [fill=black,style={circle,scale=0.1}] (e5) at (12,3) { }; 
   \node [fill=black,style={circle,scale=0.1}] (e55) at (15,3) { }; 

\path
        (n1) [-]  edge[bend left=10,thick]  node { } (e4)
        (n2) [-]  edge[bend right=10,thick]  node { } (e4)
        (e4) [->] edge[thick] node {} (e44) 
        (e44) [-] edge[bend left=90,thick]      node  {}  (n1)
        (n1) [-]  edge[bend right=10,thick]  node { } (e5)
        (n2) [-]  edge[bend left=10,thick] node { } (e5)
        (n2) [-]  edge[bend left=30,thick] node { } (e5)
        (e5) [->]  edge[thick] node { } (e55)
        (e55)  [-]  edge[bend left=90,thick] node [left] {} (n2);
 \end{tikzpicture} 
 \end{tabular}
 \caption{(Left) Feed-forward hypernetwork with three layers and auto-regulation. (Right) The quotient hypernetwork of the hypernetwork on the left by the synchrony space $\set{ x}{ x_1 = x_4 = x_6,\ x_2 = x_3 = x_5}$. }\label{hyp_ff}
 \end{figure}

\subsection{Robust synchrony subspaces} 

\newcommand{\bt}{{\mathord{\bowtie}}}
\newcommand{\tbowtie}{\mathrel{\tilde\bowtie}}
\newcommand{\tbt}{{\mathord{\tilde\bowtie}}}

Consider a hypernetwork~$(\Hc, W)$ with~$n$ cells that take their state in~$V$. Let~$\Delta\subset V^n$ be a subspace of the hypernetwork total phase space defined by equality of cell states---a \emph{polydiagonal subspace}. Define an equivalence relation~$\bowtie$ on the cells of the hypernetwork in the following way: If $x_i = x_j$ is an equality defining~$\Delta$ then $i \bowtie j$. To highlight the underlying equivalence relation, we write $\Delta = \Delta_\bt$. We say that~$\Delta_\bt$ is a \emph{hypernetwork synchrony subspace} when it is left invariant under the flow of every coupled cell system with form consistent with the hypernetwork, as defined above, that is for any admissible vector field.
In slight abuse of notation and terminology, we will forget about the phase space and call~$\Delta$ a \emph{synchrony subspace of the weighted hypergraph~$(\Hc, W)$} if it is a hypernetwork synchrony subspace for any hypernetwork on~$(\Hc, W)$. Finally, if $\Delta \subseteq \R^n$ is a polydiagonal subspace and $K \in M_{n\times n}(\R)$ leaves~$\Delta$ invariant, we also say that~$\Delta$ is a synchrony space of~$K$.

By Lemmas~\ref{lem:AdmEquiv}~and~\ref{lem:ident_split}, we have the following result.

\begin{lemma}
Two hypergraphs~$(\Hc, W)$ and $(\Hc', W')$ such that one can be obtained from the other by one (or more) purely structural splitting/combining {hyperedge} operations
 have the same set of synchrony subspaces.
\end{lemma}

Recall that for traditional coupled cell networks  there is the notion of a \emph{balanced} equivalence relation~$\bowtie$ on the set of cells~\cite{SGP03, GST05}. The balanced equivalence relations~$\bowtie$ are in one-to-one correspondence with synchrony patterns: $\Delta_\bt$ is a synchrony space for the network (that is, it is left invariant under the flow of every coupled cell system with form consistent with the network) if and only if $\bowtie$ is balanced. Motivated by the definition of balanced relation of a network introduced in~\cite{SGP03, GST05} and generalized to the weighted network setup in~\cite{ADF17, AD18}, we now define balanced equivalence relation in the hypernetwork setup.

Consider a hypernetwork $(\Hc, W)$ with set of cells~$C$ and set of hyperedges~$E$. The hypernetwork is the union of \emph{constituent hypernetworks}~$(\Hc_k, W_k)$ with identical set of cells~$C$ and  hyperedges~$E_k$ that contain the hyperedges whose tail sets have cardinality~$k$\footnote{In analogy to $k$-uniform hypergraphs, the directed hypergraphs~$\Hc_k$ can be called $k$-tail-uniform.}; note that $E_k\neq\emptyset$ if and only if $k\in\Br(\Hc)$ with $\Br(\Hc)$ as in~\eqref{eq:Orders}.
For simplicity, we will just write~$\Hc_k$ for~$(\Hc_k, W_k)$ (and $\Hc$ for $(\Hc, W)$) in the following. Trivially, the input equivalence relation of~$\Hc$ is a refinement of the input equivalence relation of every~$\Hc_k$.
 
\begin{Def} 
\label{def:bal_subhyper}
Let~$\bowtie$ be an equivalence relation on~$C$ with~$p$ equivalence classes; for a cell $c\in C$ write $\overline{c}$ for its equivalence class. Now fix an ordering of the $\bowtie$-classes, say $(\overline{c}_1, \dotsc, \overline{c}_p)$, where $c_i\in C$ for $i=1, \dotsc, p$.
Fix~$k\in\Br(\Hc)$ and consider $e\in E_k$ with weight~$w_e$.\\
(i)~The \emph{pattern determined by $\bowtie$ on~$e$} is a vector with~$p$ nonnegative integer entries, $\overrightarrow{m}(e) = (m_1, \dotsc, m_p)$, whose coefficients~$m_i$ indicates the number of cells at the tail set~$T(e)$ of~$e$ which are in the class~$\overline{c}_i$. Thus, as $e \in E_k$, we have that $\sum_{i=1}^p m_i = k$ and some of the~$m_i$ can be zero.\\
(ii)~If $c \in C$ and $e\in\BS_k(c)$ has pattern $\overrightarrow{m}(e)$ determined by~$\bowtie$, the \emph{weight of the pattern $\overrightarrow{m}(e)$ on the cell~$c\in C$ determined by~$\bowtie$} is the sum of the weights of the hyperedges~$e'\in\BS_k(c)$ with $\overrightarrow{m}(e')=\overrightarrow{m}(e)$ determined by~$\bowtie$.\\
(iii)~We say that~$\bowtie$ is \emph{balanced} for the constituent hypernetwork~$\Hc_k$ if for every two distinct cells $c,c'\in C$ such that $c \bowtie c'$, the set of patterns determined by the hyperedges of the sets~$\BS(c)$ and~$\BS(c')$ coincide and each pattern has the same pattern weight on both cells. 
\hfill $\Diamond$
\end{Def}

\begin{Def} 
Consider a hypernetwork~$\Hc$ with cells~$C$, hyperedges~$E$, and constituent hypernetworks~$\Hc_k$ as defined above.
Let~$\bowtie$ be an equivalence relation on~$C$ refining~$\simI$. We say that~$\bowtie$ is \emph{balanced} if it is balanced for every constituent hypernetwork~$\Hc_k$.
\hfill $\Diamond$
\end{Def}

Note that input equivalence is not always a balanced relation; this was already noted by  Stewart~\cite[Section 6]{S07} for standard $n$-node directed graphs. That is, the coarsest balanced equivalence relation refines~$\simI$ but does not need not to coincide with~$\simI$. See also Aldis~\cite{A08} for the description of a polynomial-time algorithm to compute the coarsest balanced equivalence relation of a graph. Since it is a necessary condition for an equivalence relation on the nodes to be balanced is to refine~$\simI$, we include that assumption at the above definition. The coarsest partition corresponds to the most synchrony that is possible.

\begin{rem}
(i) The finest partition where each cell is only equivalent to itself (the equivalence classes are singletons) is trivially balanced. The corresponding synchrony subspace is the entire phase space; the finest partition corresponds to the least synchrony.\\
(ii) The relation with just a single equivalence class (the coarsest partition possible) is balanced if all cells are input equivalent. Indeed, if there is only one equivalence class then for any hyperedge $e\in E(\Hc_k)$ we have only one pattern $\overrightarrow{m}(e) = (k)$. Thus, condition (ii) in Definition~\ref{def:bal_subhyper} for a relation to be balanced is equivalent to condition~(iib) in Definition~\ref{def:InputEquiv} for input equivalence. Since the associated synchrony subspace corresponds to full synchrony, this gives an explicit condition for the existence of full synchrony as an invariant subspace.
\end{rem}

\begin{exam} \label{example_b_notb}
(i) Consider the directed hypernetwork in Figure~\ref{hyp_bal} with node set $C = \{ 1, 2, \dotsc, 14\}$. All the hyperedges have tail set of cardinality $3$ and so $\Hc = \Hc_3$. Moreover, all the cell backward stars are empty, except for cells $4$ and $14$. As $\sum_{e \in \BS(4)} w_e = 2 + 1 = 3$ coincides with  $\sum_{e \in \BS(14)} w_e = 1 + 1 + 1 =3$, we have that $4 \simI 14$, and so the classes of the input relation $\simI$ are $ \{4,14\}$ and $C \setminus \{4,14\}$. Note that in this case~$\simI$ is balanced. Consider now the equivalence $\bowtie$ on $C$ with classes 
\[\overline{1} = \{ 1,5,6,8,9, 11\},\, \overline{2} = \{ 2,3,7,10,12,13\},\, \overline{4} = \{4,14\}\, .\] 
In Figure~\ref{hyp_bal}, cells in the class~$\overline{1}$ have white color, cells in the class~$\overline{2}$ have blue color, and those in the class~$\overline{4}$ have pink color. 
Consider the equivalence classes ordered as $\left(\overline{1}, \overline{2}, \overline{4} \right)$. We have that~$\bowtie$ determines two types of patterns, $(2,1,0)$ and $(1,2,0)$, for the hyperedges in both~$\BS(4)$ and~$\BS(14)$. The pattern $(2,1,0)$ corresponds to a hyperedge with tail set consisting of two white cells and one blue cell; the pattern $(1,2,0)$ corresponds to a hyperedge whose tail set has two blue cells and one white cell. 
For cell~$4$, the incoming hyperedge with pattern~$(2,1,0)$ has weight~$1$ and the hyperedge with pattern $(1,2,0)$ has weight~$2$. 
For cell~$14$, there are two hyperedges in~$\BS(14)$ with pattern $(1,2,0)$ with weight~$1$ each, and there is a hyperedge with pattern $(1,2,0)$ with weight~$1$. It follows that for both cells~$4$ and~$14$ the pattern $(1,2,0)$ has pattern weight~1 and $(2,1,0)$ has pattern weight~$2$. Thus $\bowtie$ is balanced.\\
(ii) For the hypernetwork in Figure~\ref{hyp_notbal}, with node set $C = \{ 1, 2, \dotsc, 12\}$, the input relation $\simI$ has also two classes, $\{4,12\}$ and $C \setminus \{4,12\}$, and is balanced. Consider the refined equivalence~$\bowtie$ on~$C$ with classes 
\[\overline{1} = \{ 1,2,8,9\},\, \overline{3} = \{ 3,5,6,7,10,11\},\, \overline{4} = \{4,12\},\] 
which is not balanced as we will now show. First, note that all the hyperedges have tail set with cardinality~$3$ and all the cell backward stars are empty, except for cells~$4$ and~$12$. 
Second, for the ordering $\left( \overline{1}, \overline{3} , \overline{4} \right)$ of the $\bowtie$-classes, we have that for cell~$4$, the hyperedges in~$\BS(4)$ have patterns $(0,3,0)$ and $(3,0,0)$. For cell $12$, the hyperedges in~$\BS(12)$ have two types of patterns $(2,1,0)$ and $(1,2,0)$. Thus $\bowtie$ is not balanced.  

\hfill $\Diamond$
\end{exam}

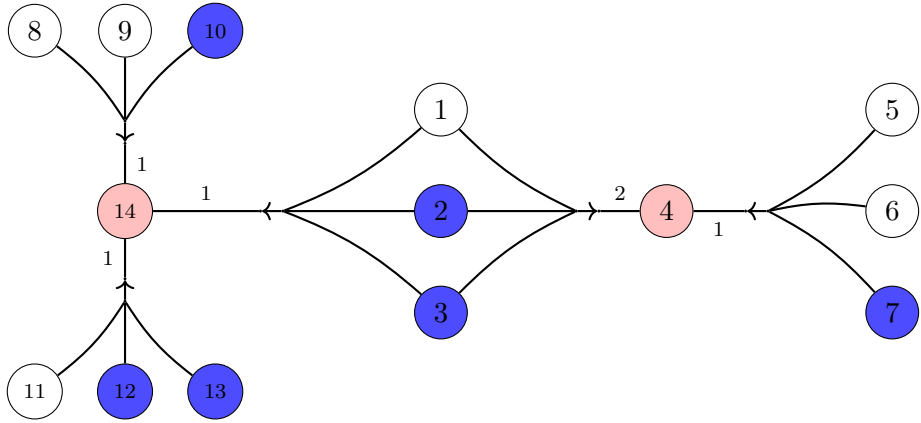
\begin{figure}
\begin{tikzpicture}
 [scale=.15,auto=left, node distance=1.5cm, 
 ]
 \node[fill=pink,style={circle,draw}] (n14) at (-24,0) {\tiny{14}};
 \node[fill=white,style={circle,draw}] (n8) at (-32,16)  {\small{8}};
 \node[fill=white,style={circle,draw}] (n9) at (-24,16)  {\small{9}};
 \node[fill=blue!70,style={circle,draw}] (n10) at (-16,16)  {\tiny{10}};
 \node[fill=white,style={circle,draw}] (n11) at (-32,-16)  {\tiny{11}};
 \node[fill=blue!70,style={circle,draw}] (n12) at (-24,-16)  {\tiny{12}};
 \node[fill=blue!70,style={circle,draw}] (n13) at (-16,-16)  {\tiny{13}};
 \node[fill=white,style={circle,draw}] (n1) at (4,9)  {\small{1}};
 \node[fill=blue!70,style={circle,draw}] (n2) at (4,0) {\small{2}};
 \node[fill=blue!70,style={circle,draw}] (n3) at (4,-9)  {\small{3}};
 \node[fill=pink,style={circle,draw}] (n4) at (24,0) {\small{4}};
  \node[fill=white,style={circle,draw}] (n5) at (44,9)  {\small{5}};
 \node[fill=white,style={circle,draw}] (n6) at (44,0)  {\small{6}};
  \node[fill=blue!70,style={circle,draw}] (n7) at (44,-9)  {\small{7}};
 \node [fill=black,style={circle,scale=0.1}] (e1) at (-10,0) { };
  \node [fill=black,style={circle,scale=0.1}] (e11) at (-12,0) { };
   \node [fill=black,style={circle,scale=0.1}] (e2) at (-24,-8) { };
    \node [fill=black,style={circle,scale=0.1}] (e22) at (-24,-6) { };
  \node [fill=black,style={circle,scale=0.1}] (e3) at (33,0) { };
  \node [fill=black,style={circle,scale=0.1}] (e33) at (31,0) { };
  \node [fill=black,style={circle,scale=0.1}] (e4) at (-24,8) { };
   \node [fill=black,style={circle,scale=0.1}] (e44) at (-24,6) { };
  \node [fill=black,style={circle,scale=0.1}] (e5) at (16,0) { }; 
   \node [fill=black,style={circle,scale=0.1}] (e55) at (18,0) { }; 

\path
        (n1) [-]  edge[bend right=10,thick] node { } (e5)
        (n2) [-]  edge[thick]  node { } (e5)
         (n3) [-]  edge[bend left=10,thick]  node { } (e5)
          (e5) [->] edge[thick] node {}  (e55)       
         (e55) [-] edge[thick] node[above] {{\tiny $2$}}   (n4)       
        (n1) [-]  edge[bend left=10,thick] node { } (e1)
        (n2) [-]  edge[thick]  node { } (e1)
         (n3) [-]  edge[bend right=10,thick]  node { } (e1)
         (e1) [->] edge[thick] node {} (e11)     
         (e11) [-] edge[thick] node[above]  {{\tiny $1$}}   (n14)     
        (n11) [-]  edge[bend right=10,thick] node { } (e2)
        (n12) [-]  edge[thick]  node { } (e2)
         (n13) [-]  edge[bend left=10,thick]  node { } (e2)
          (e2) [->] edge[thick] node {}   (e22)   
         (e22) [-] edge[thick] node {{\tiny $1$}}   (n14)      
        (n8) [-]  edge[bend left=10,thick] node { } (e4)
        (n9) [-]  edge[thick]   node { } (e4)
         (n10) [-]  edge[bend right=10,thick]  node { } (e4)
         (e4) [->] edge[thick] node {} (e44)    
         (e44) [-] edge[thick] node {{\tiny $1$}}   (n14)      
         (n5) [-]  edge[bend left=10,thick] node { } (e3)
         (n6) [-]  edge[bend right=10,thick]  node { } (e3)
         (n7) [-]  edge[bend right=10,thick]  node { } (e3)
         (e3) [->] edge[thick] node {}  (e33)   
         (e33) [-] edge[thick] node {{\tiny $1$}}   (n4)      
;
 \end{tikzpicture}  
 \caption{The equivalence relation with three classes represented by the three colours is balanced for the hypernetwork.}\label{hyp_bal}
 \end{figure}

\begin{figure}
\begin{tikzpicture}
 [scale=.15,auto=left, node distance=1.5cm, 
 ]
  \node[fill=white,style={circle,draw}] (n9) at (-36,9)  {\small{9}};
 \node[fill=blue!70,style={circle,draw}] (n10) at (-36,0) {\tiny{10}};
 \node[fill=blue!70,style={circle,draw}] (n11) at (-36,-9)  {\tiny{11}};
 \node[fill=pink,style={circle,draw}] (n12) at (-24,0) {\tiny{12}};
 \node[fill=white,style={circle,draw}] (n8) at (-6,-10)  {\small{8}};
 \node[fill=white,style={circle,draw}] (n1) at (-12,9)  {\small{1}};
 \node[fill=white,style={circle,draw}] (n2) at (-12,0) {\small{2}};
 \node[fill=blue!70,style={circle,draw}] (n3) at (-12,-9)  {\small{3}};
 \node[fill=pink,style={circle,draw}] (n4) at (0,0) {\small{4}};
  \node[fill=blue!70,style={circle,draw}] (n5) at (12,9)  {\small{5}};
 \node[fill=blue!70,style={circle,draw}] (n6) at (12,0)  {\small{6}};
  \node[fill=blue!70,style={circle,draw}] (n7) at (12,-9)  {\small{7}};

 \node [fill=black,style={circle,scale=0.1}] (e1) at (-18,0) { };
 \node [fill=black,style={circle,scale=0.1}] (e11) at (-20,0) { };
  \node [fill=black,style={circle,scale=0.1}] (e3) at (6,0) { };
   \node [fill=black,style={circle,scale=0.1}] (e33) at (4,0) { };
  \node [fill=black,style={circle,scale=0.1}] (e5) at (-6,0) { }; 
   \node [fill=black,style={circle,scale=0.1}] (e55) at (-4,0) { }; 
  \node [fill=black,style={circle,scale=0.1}] (e4) at (-29,0) { };
  \node [fill=black,style={circle,scale=0.1}] (e44) at (-27,0) { };
\path
        (n9) [-]  edge[bend right=10,thick] node { } (e4)
        (n10) [-]  edge[thick]   node { } (e4)
         (n11) [-]  edge[bend left=10,thick]  node { } (e4)
          (e4) [->] edge[thick] node {}(e44)    
         (e44) [-] edge[thick] node {{\tiny $1$}}   (n12)      
        (n1) [-]  edge[bend left=10,thick] node { } (e1)
        (n2) [-]  edge[thick]  node { } (e1)
         (n3) [-]  edge[bend right=10,thick]  node { } (e1)
          (e1) [->] edge[thick] node {}(e11)     
         (e11) [-] edge[thick] node {{\tiny $1$}}   (n12)           
        (n1) [-]  edge[bend right=20,thick] node { } (e5)
        (n2) [-]  edge[bend right=20, thick]  node { } (e5)
         (n8) [-]  edge[bend left=20,thick]  node { } (e5)
         (e5) [->] edge[thick] node {}   (e55)   
         (e5) [-] edge[thick] node {{\tiny $1$}}   (n4)          
         (n5) [-]  edge[bend left=10,thick] node { } (e3)
         (n6) [-]  edge[bend right=10,thick]  node { } (e3)
         (n7) [-]  edge[bend right=10,thick]  node { } (e3)
         (e3) [->] edge[thick] node {} (e33)   
         (e33) [-] edge[thick] node {{\tiny $1$}}   (n4)      
;
 \end{tikzpicture}  
 \caption{The equivalence relation with three classes represented by the three colours is not balanced for the hypernetwork.}\label{hyp_notbal}
 \end{figure}
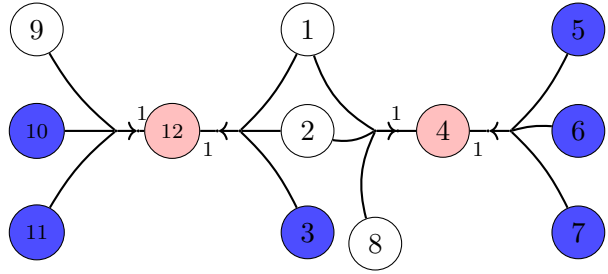
 
 \begin{prop}
 The definition of balanced equivalence relation for hypernetworks includes, as a particular case, the definition of balanced equivalence relation for networks. 
 \end{prop}
 
 \begin{proof} 
Let~$\Hc$ be a hypernetwork which is a network, that is, the tail sets of all the hyperedges have cardinality~$1$. Thus~$\Hc_1=\Hc$. 
Given an equivalence relation $\bowtie$ on the set of cells of the network~$\Hc$, we have then to consider Definition~\ref{def:bal_subhyper}. Let~$p$ be the number of $\bowtie$-classes and fix an ordering of those classes, say $(\overline{c}_1, \dotsc, \overline{c}_p)$.
For every edge~$e$ in~$\Hc$, the pattern determined by $\bowtie$ on~$e$, $\overrightarrow{m}(e)$,  is a vector with one entry equal to~$1$ and all the other $p-1$ entries equal to~$0$. 
For a cell~$c$ and an edge~$e$ with $H(e) = \{c\}$ if the the $i$th entry is the nonzero entry of the pattern $\overrightarrow{m}(e)$ determined by $\bowtie$ then the pattern weight of the pattern  $\overrightarrow{m}(e)$ on the cell $c$ is the sum of the weights of the edges with $H(e) = \{c \}$ that have the same pattern $\overrightarrow{m}(e)$, that is, the sum of the weights of the edges with ${H}(e) = \{c\}$ and $T(e) \in \overline{c}_i$.
Then, by Definition~\ref{def:bal_subhyper}, $\bowtie$ is balanced for the network~$\Hc$ when, for every two distinct cells $c,c'\in C$ such that $c \bowtie c'$, the pattern sets  determined by the edges of the sets~$\BS(c)$ and~$\BS(c')$ coincide, that is, the pattern set  determined by the edges with $H(e) = \sset{c}$ coincides with the pattern set  determined by the edges with $H(e) = \sset{c'}$, which means that cell $c$ receives edges from cells in the class~$\overline{c}_i$ if and only if 
cell~$c'$ also receives edges from cells in that class. Moreover, each pattern has the same pattern weight on both cells, which means that the sum of the weights of the edges from cells in class~$\overline{c}_i$ to cell~$c$ equals the sum of the weights of the edges from cells in class~$\overline{c}_i$ to cell~$c'$. 
\end{proof}

\subsection{Quotients}\label{sec:Quotients}

Given a weighted directed hypergraph~$(\Hc, W)$ and a balanced equivalence relation~$\bowtie$ on the cells, we now define the quotient of~$(\Hc, W)$ with respect to~$\bt$. The quotient describes the admissible vector fields for~$(\Hc, W)$ when restricted to the synchrony space~$\Delta_{\bowtie}$. To keep notation simple, we assume---without loss of generality by Lemma~\ref{lem:AdmEquiv}---that all hyperedges in~$E(\Hc)$ have tails of cardinality one.

\begin{Def} 
Let~$\Hc$ be a hypernetwork with cells~$C$ and hyperedges~$E$ (whose heads have cardinality one by assumption). 
Let~$\bowtie$ be a balanced equivalence relation on~$C$ with~$p$ classes, say $\overline{C}=(\overline{c}_1, \dotsc, \overline{c}_p)$.\\
(i)~Let $e\in E(\Hc)$ be a hyperedge with head~$\sset{c}$ and pattern $\overrightarrow{m}(e) = (m_1, \dotsc, m_p)$ onto~$c$. The \emph{projected hyperedge~$\overline{e}$ with respect to~$\bt$} has head $H(\overline{e}) = \sset{\overline{c}}$ (where~$\overline{c}$ denotes the equivalence class of~$c$) and tail multiset\footnote{Note that repeated entries are maintained for the tail of~$\bar{e}$ as it is a multiset.}
\[T(\overline{e}) = \tsset{\underbrace{\overline{c}_1, \dotsc, \overline{c}_1}_{m_1 \text{ times}}, \underbrace{\overline{c}_2, \dotsc, \overline{c}_2}_{m_2 \text{ times}}, \dotsc, \underbrace{\overline{c}_p, \dotsc, \overline{c}_p}_{m_p \text{ times}}}.\]
The \emph{weight~$\overline{w}$ of~$\overline{e}$} is the pattern weight~$w$ of~$\overrightarrow{m}(e)$.\\
(ii) Let~$\overline{E}$~the hyperedges defined in~(i) and~$\overline{W}$ the corresponding weights. Write $\overline{\Hc} = (\overline{C}, \overline{E})$. The \emph{quotient of~$\Hc$ by~$\bowtie$}, is the hypernetwork $\Hc/\bt := (\overline{\Hc}, \overline{W})$.  
\hfill $\Diamond$
\end{Def}

By definition, all hyperedges of~$\overline{\Hc}$ have a head of cardinality one. For a cell $\bar c$ of~$\overline{\Hc}$, the backward star~$\BS(\overline{c})$ is formed by the hyperedges~$\overline{e}$ derived from each distinct pattern determined by~$\bowtie$ in~$\BS(c)$.

\begin{rem} 
Recall that different hypernetworks (with distinct underlying hypergraphs) can be identical as coupled cell systems (see Lemma~\ref{lem:AdmEquiv}).\\
(i)~Any hypernetworks that are identical to each other as coupled cell networks via Lemma~\ref{lem:AdmEquiv} have the same quotient, while their incidence digraph differs in general.\\
(ii)~The quotient~$\Hc/\bt = (\overline{\Hc}, \overline{W})$ may be equivalent as a coupled cell network to a different hypernetwork~$(\overline{\Hc}', \overline{W}')$ (for example, by combining edges that have the same tail set).
However, in our context the quotient is uniquely defined by the convention that the hyperedges in the quotient will have a head of cardinality one. 
\hfill $\Diamond$
\end{rem}

\begin{exam}
Consider the directed hypergraph $\Hc=(C,E)$ on the left in Figure~\ref{figure_example1}. Thus $C = \sset{1, \dotsc, 6}$ and 
\begin{align*}
e_1 &= \left( \{ 2,5\}, \{1 \} \right), & e_2 &= \left( \{ 2\}, \{2,4 \} \right), & e_3 &= \left( \{ 1,2\}, \{6 \} \right),\\
e_4 &= \left( \{ 4,6\}, \{3, 5 \} \right), & e_5 &= \left( \{ 4\}, \{3 \} \right).
\end{align*}
where each edge has weight~$w_e = 1$. 
The resulting hypernetwork is identical as a coupled cell system to the hypernetwork with underlying hypergraph $\Hc'=(C,E')$ such that the head~$H(e)$ of any hyperedge~$e\in E'$ has cardinality 1. Specifically, by splitting the head sets of hyperedges $e_2$ and $e_4$ we have
\begin{align*}
E' = \sset{ e_1, \left( \sset{2}, \sset{2} \right), \left(\sset{2}, \sset{4} \right), e_3, \left( \sset{4,6}, \sset{3} \right), \left( \sset{4,6}, \sset{5} \right), e_5}\, .
\end{align*}
By assumption in the beginning of this section, we will identify $\Hc=(C,E)$ with $\Hc'=(C,E')$ and drop the~$'$. 

For the balanced coloring indicated by the shading of the nodes in Figure~\ref{figure_example1}, the cells of the quotient are given by the equivalence classes 
\[
\overline{C} = \sset{\overline{1}=\sset{1,5,6}, \overline{2}=\sset{2,4}, \overline{3}=\sset{3}}\, .
\] 
The sets $\BS(\overline{1}), \BS(\overline{2}), \BS(\overline{3})$ are obtained from $\BS(1), \BS(2)$ and $\BS(3)$, respectively, and thus
\[
\overline{E} = \sset{
\left(\sset{\overline{1},\overline{2}}, \sset{\overline{1}} \right),
\left(\sset{\overline{2}}, \sset{\overline{2}} \right), 
\left(\sset{\overline{1},\overline{2}}, \sset{\overline{3}} \right),
\left(\sset{\overline{2}}, \sset{\overline{3}} \right)
},
\]
all with weight equal to~$1$.
Note that~$\Hc/\bt$ is identical as coupled cell hypernetwork to the hypernetwork shown in Figure~\ref{figure_example1} to the right.
\hfill $\Diamond$
\end{exam}

\begin{thm}
Suppose that $(\Hc, W)$ is a hypernetwork and~$\bt$ is a balanced equivalence relation on~$(\Hc, W)$. The quotient~$\Hc/\bt = (\overline{\Hc}, \overline{W})$ is well defined. Moreover, the dynamics of~$(\Hc, W)$ restricted to~$\Delta_\bt$ correspond to the evolution of the coupled cell hypernetwork~$\Hc/\bt$.
\end{thm}

\begin{proof}
The first assertion follows from the definition of a balanced equivalence relation: An equivalence relation is balanced exactly when the weight of a pattern is the same for all cells in the same equivalence class.
The second assertion follows from the construction of the quotient: (a)~The heads of the hyperedges~$\overline{e}$ in the quotient identify synchronized cells and (b)~the weights of the edges in the quotient sum---for a fixed head---the weights of the corresponding edges with the same pattern.
\end{proof}

\begin{rem}
The (somewhat nonstandard) convention to allow multisets as tails of directed hyperedges becomes essential in the coupled cell hypernetwork formalism presented in this work that considers generic features for all admissible vector fields simultaneously. By contrast, if one considers a specific hypernetwork coupling $(\Hc, W, Q)$, then one may be able to identify edges whose tail sets have cardinality~$k$ with edges with lower tail set cardinalities. For example, consider cells whose phase space is~$\R$ and hypergraph coupling with $Q_2(x_1; x_2, x_3) = x_2x_3$, $Q_1(x_1; x_2)=x_2^2$. If $2\bowtie 3$ the quotient of the edge $e=(\sset{2,3}, \sset{1})$ can be identified with an edge $e'=(\sset{2}, \sset{1})$ of the same weight.
\hfill $\Diamond$
\end{rem}

\begin{exam} Recall the hypernetwork $\Hc=\Hc_3$ in Figure~\ref{hyp_bal} and the balanced equivalence relation $\bowtie$ with classes $\overline{1} = \{ 1,5,6,8,9, 11\},\, \overline{2} = \{ 2,3,7,10,12,13\},\, \overline{4} = \{4,14\}$. The quotient network $\Hc/\bt$ has set of nodes $\overline{1}, \overline{2} , \overline{4}$ and $\BS( \overline{4})$ is formed by two hyperedges from the two distinct patterns determined by $\bowtie$ in $\BS(4)$ as described in Example~\ref{example_b_notb}; see Figure~\ref{hyp_quo}.
\hfill $\Diamond$
\end{exam}

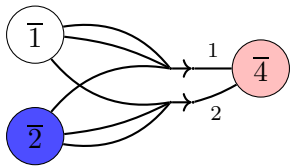
\begin{figure}
\begin{tikzpicture}
 [scale=.15,auto=left, node distance=1.5cm, 
 ]
 \node[fill=white,style={circle,draw}] (n1) at (4,9)  {\small{$\overline{1}$}};
 \node[fill=blue!70,style={circle,draw}] (n2) at (4,0) {\small{$\overline{2}$}};
 \node[fill=pink,style={circle,draw}] (n4) at (24,6) {\small{$\overline{4}$}};
   \node [fill=black,style={circle,scale=0.1}] (e3) at (16,6) { };
    \node [fill=black,style={circle,scale=0.1}] (e33) at (18,6) { };
  \node [fill=black,style={circle,scale=0.1}] (e5) at (16,3) { }; 
   \node [fill=black,style={circle,scale=0.1}] (e55) at (18,3) { };
\path
        (n1) [-]  edge[bend right=30,thick] node { } (e5)
        (n2) [-]  edge[bend right=10,thick]   node { } (e5)
        (n2) [-]  edge[bend right=30,thick]  node { } (e5)
         (e5) [->] edge[thick] node [below] {}  (e55)     
            (e55) [-] edge[bend right=7, thick] node [below] {{\tiny $2$}}   (n4)   
         (n1) [-]  edge[bend left=10,thick] node { } (e3)
         (n1) [-]  edge[bend left=30,thick]  node { } (e3)
         (n2) [-]  edge[bend  left=30,thick]  node { } (e3)
         (e3) [->] edge[thick] node {}  (e33)  
         (e33) [-] edge[thick] node {{\tiny $1$}}   (n4)      
;
 \end{tikzpicture}  
 \caption{The quotient hypernetwork of the hypernetwork in Figure~\ref{hyp_bal} by the balanced equivalence relation on the set of nodes whose classes are represented by the three colours.}\label{hyp_quo}
 \end{figure}

\begin{thm} \label{thm:sync_bal}
Let~$\Hc$ be a weighted directed hypergraph on the node set $C = \{1, 2, \dotsc, n\}$ and hyperedge set~$E$. An equivalence relation~$\bowtie$ on the node set is balanced if and only if for any hypernetwork associated with  $\Hc$, the polydiagonal space~$\Delta_\bt$ defined in terms of the equalities on the cell coordinates~$x_i$, for $i \in C$, determined by $\bowtie$,  is a synchrony space of~$\Hc$.  
\end{thm}

\begin{proof}
By definition of~$\bowtie$ being balanced, it follows that if~$\bowtie$ is balanced then~$\Delta_\bt$ is a synchrony space of~$\Hc$. Now, if~$\Delta_\bt$ is a synchrony space of~$\Hc$, then in particular, we can consider the admissible equations  
where all the internal cell phase spaces  are $\R$ and the coupling functions~$Q_k$ have the form
\[
Q_k(x_0;x_1, x_2, \dotsc, x_k) = x_1 x_2 \cdots x_k\, . 
\]
Consider the decomposition of $H$ into its constituent hypernetworks~$\Hc_k$, for $k=j_1, \dotsc, j_r$ according to the (positive and integer) cardinalities~$k$  of the tail sets of its hyperedges.  
Given two distinct cells~$c,c'$ such that $x_c = x_{c'}$ is one of the  equalities defining~$\Delta_\bt$, we have that the corresponding cell equations, at the restriction to~$\Delta_\bt$ have to coincide. The restriction of the cells~$c$ and~$c'$ equations, are so polynomials which are each the sum of homogeneous polynomials of degrees $j_1, \dotsc, j_r$. Thus 
the two polynomials coincide if and only if they coincide degree by degree. (Equivalently, if and only if $\Delta_\bt$ is a synchrony space of each constituent hypernetwork~$\Hc_k$.) 
For a fixed degree~$k$, then each distinct monomial that is appearing at the equation for cell~$c$, it has also to appear at equation for cell~$c'$, and with the same coefficient. Now each monomial of the~$c$ equation ($c'$~equation)  with coefficient~$m_c$ ($m_{c'}$) corresponds to a pattern~$\overrightarrow{m}(e_c)$ ($\overrightarrow{m}(e_{c'})$) determined by~$\bowtie$ at the hyperedges in~$\BS(c)$ ($\BS(c')$) with weight  $m_c$ ($m_{c'}$). Thus the set of the distinct patterns determined by~$\bowtie$ in~$\BS(c)$ and~$\BS(c')$ must coincide, and the corresponding multiplicities have also to coincide. That is, $\bowtie$ is balanced. 
\end{proof}

Trivially, we have the following result.

\newcommand{\Qc}{\mathcal{Q}}

\begin{thm}\label{thm:dyn_quot}
Let~$\Hc$ be a weighted directed hypernetwork on the node set $C = \{1, 2, \dotsc, n\}$ and hyperedge set~$E$. Let~$\bowtie$ be a balanced equivalence relation on~$C$. Let~$\Qc$ be the quotient hypernetwork~$\Hc/\bt$. Then:\\
(i) Any coupled cell system consistent with~$\Hc$ restricted to~$\Delta_\bt$ is a coupled cell system consistent with the quotient hypernetwork~$\Qc$. \\
(ii) Any coupled cell system consistent with the hypernetwork~$\Qc$ is the restriction of a coupled cell system consistent with the hypernetwork~$\Hc$ restricted to~$\Delta_\bt$.  
\end{thm}

\begin{exam}
Consider the hypernetwork~$\Hc$ in Figure~\ref{hyp_bal} and the balanced equivalence relation~$\bowtie$ presented in Example~\ref{example_b_notb}(i). Consider coupled cell systems consistent with~$\Hc$, where the cell phase space is~$V$, the internal dynamics is given by~$f: V \to V$ and the coupling by $Q_3:\, V^4 \to V$. Since the equivalence relation~$\bowtie$ is balanced, then the polydiagonal space
\begin{align*}
\Delta_\bt = \rset{x}{\begin{array}{l}
 x_1 = x_5 = x_6 = x_8 = x_9 = x_{11},\\
 x_2 = x_3= x_7 = x_{10} = x_{12} = x_{13}, x_4 = x_{14} 
 \end{array}
 }
\end{align*}
is a synchrony space of~$\Hc$, that is, equations for~$\Hc$ leave~$\Delta_\bt$ invariant. 
The restriction of those equations to~$\Delta_\bt$ gives rise to coupled cell systems consistent with the quotient hypernetwork~$\Hc/\bt$ in Figure~\ref{hyp_quo} with cells evolving according to
\begin{align*}
\dot{x}_1 &= f(x_1),\\
\dot{x}_2 &= f(x_2),\\
\dot{x}_4 &= f(x_4) + Q_3(x_4, x_1, x_1, x_2) + 2 Q_4(x_4, x_1, x_2, x_2)\, .
\end{align*}
\hfill $\Diamond$
\end{exam}

\begin{rem}
Due to Lemma~\ref{lem:ident_split}, the results in Theorem~\ref{thm:dyn_quot}, concerning the restriction of the dynamcs to the synchrony subspace $\Delta_\bt$, apply to every hypernetwork obtained from the quotient hypernetwork $\Qc$ by one (or more) purely structural combining {hyperedge} operations, since they are identical as coupled cell systems.
\hfill $\Diamond$
\end{rem}

\subsection{Robust synchrony subspaces via the incidence digraph} 
In the previous section, we established the notion of a balanced relation for a hypernetwork~$\Hc$. At the same time, as outlined in Section~\ref{Sec:hypergraph}, the hypergraph~$\Hc$ can also be represented as a bipartite graph~$\Dc_{\cH}$ (cf.~Definition~\ref{def:bidigraph}) for which traditional notions of balanced relations and synchrony subspaces apply. How do the hypergraph synchrony subspaces of~$\Hc$ and the synchrony subspaces of $\Dc_\Hc$ relate? 
We now show how to find the set (lattice) of the synchrony subspaces for an hypernetwork~$(\cH,W)$ using the associated incidence digraph~$\Dc_{\cH}$ of~$\cH$ with nodes given by the nodes and hyperedges of $\cH$. More concretely, we prove that the synchrony subspaces for the hypernetwork $(\cH, W)$ can be obtained by a `projection' of  the synchrony subspaces of  the adjacency matrix of the incidence digraph $\Dc_{\cH}$.

We start by relating the set of balanced equivalence relations on the set of cells of an hypernetwork $(\cH,W)$ with those on the set of the nodes of its incidence digraph~$\Dc_{\cH}$. 

\begin{Def} \label{def:equiv_rel}
Let $(\Hc,W)$ be an hypernetwork with cells~$C=C(\Hc)$ and hyperedges~$E=E(\Hc)$, and let~$\Dc_{\Hc}$ be the corresponding incidence digraph with nodes $C(\Dc_{\Hc}) = C(\Hc) \cup E(\Hc)$.

\noindent (i) Given an equivalence relation $\bowtie$ on $C$ for $\cH$, we define the equivalence relation $\bowtie_\Dc$ on $C \cup E$ for $\Dc_{\cH}$ in the following way: 
\begin{itemize}
\item[(a)] $c \bowtie_\Dc c'$ iff $c \bowtie c'$, for $c,c' \in C$;
\item[(b)] $e_i \bowtie_\Dc e_j$ iff $\overrightarrow{m} (e_i) = \overrightarrow{m} (e_j)$, for $e_i, e_j \in E$. 
\end{itemize}
with $\overrightarrow{m} (e)$ the pattern determined by $\bowtie$ on the hyperedge $e$.

\noindent (ii) Given an equivalence relation $\tbt$ on $C \cup E$ for $\Dc_{\cH}$, we define the equivalence relation $\tbt_{\cH}$ on $C$ for $\cH$ through
\begin{itemize}
\item[(a)] $c \tbowtie_{\cH} c'$ iff $c\tbowtie c'$, for $c,c' \in C$.
\end{itemize}
We say that the relation~$\tbt_{\cH}$  is the {\em projection} of the relation $\tbt$.
\hfill $\Diamond$
\end{Def}

Given the definition above, we have then the following result.

\begin{thm} \label{thm:bij_bal}
Let $(\cH,W)$ be an hypernetwork and $\Dc_{\cH}$ the corresponding incidence digraph. We have: \\
(i) For each balanced equivalence relation $\bowtie$ for $(\cH,W)$ the corresponding equivalence relation $\bowtie_\Dc$ for~$\Dc_{\cH}$ is also balanced; \\
(ii) Each balanced equivalence relation~$\tbt$ for $\Dc_{\cH}$ projects into a balanced equivalence relation $\tbt_{\cH}$  for $(\cH,W)$.
\end{thm}

\begin{proof}
Let $(\cH,W)$ be an hypernetwork with cells~$C$ and hyperedges~$E$, and let~$\Dc_{\cH}$ be the associated incidence digraph with nodes $C \cup E$.

(i) Let $\bowtie$ be a balanced equivalence relation on the set of cells $C$ of the hypernetwork $(\cH,W)$ and consider the corresponding equivalence relation $\bowtie_\Dc$ on the set of nodes $C \cup E$ of the bipartite network $\Dc_{\cH}$, as in Definition~\ref{def:equiv_rel}. By definition, two nodes $e_i, e_j \in E$ of $\Dc_{\cH}$ such that $e_i \bowtie_\Dc e_j$ correspond to two hyperedges $e_i$ and $e_j$ of $\cH$ that have the same pattern determined by $\bowtie$. Also, note that  the input set of a node $e_i \in E$ of~$\Dc_{\cH}$ corresponds to the tail~$T(e_i)$ of the hyperedge~$e_i$ in~$\cH$.
Thus: (a) for every two nodes $e_i, e_j \in E$ of $\Dc_{\cH}$ such that $e_i \bowtie_\Dc e_j$ there is a bijection between their input sets in $\Dc_{\cH}$ preserving the $\bowtie_\Dc$-classes. 
Consider now two nodes $c, d \in C$ of $\Dc_{\cH}$ such that $c \bowtie_\Dc d$, and thus with $c \bowtie d$. Then, since $\bowtie$ is balanced, the pattern sets determined by the hyperedges of the sets~$\BS(c)$ and~$\BS(d)$ coincide and each pattern has the same weight on both cells. Note that the input set $I_B(i)$ of a node $i \in C$ of $\Dc_{\cH}$ is given by the backward stars $\BS(i)$ of $i$ in $\cH$. We have then: (b) for any two nodes $c, c' \in C$ such that $c \bowtie_\Dc c'$, for every $\bowtie_\Dc$-class, the sum of the weights of the edges in $\Dc_{\cH}$ directed to nodes~$c$ and~$c'$, from the nodes in that $\bowtie_\Dc$-class, is the same. From~(a) and~(b), it follows that the equivalence relation~$\bowtie_\Dc$, as defined in Definition~\ref{def:equiv_rel}, is balanced. Thus, we have shown that, for every balanced equivalence relation~$\bowtie$ for the hypernetwork $(\cH,W)$, we can associate a balanced equivalence relation~$\bowtie_\Dc$ for the incidence digraph~$\Dc_{\cH}$.

(ii) Let $\tbt$ be a balanced equivalence relation on the set of nodes $C \cup E$ for the incidence digraph $\Dc_{\cH}$ and consider the equivalence relation $\tbt_{\cH}$ that is a projection on the set of cells~$C$ of~$\cH$ satisfying $c \tbowtie_{\cH} c'$ if and only if $c \tbowtie c'$. Since $\tbt$ is balanced, for $c, c' \in C$, if $c \tbowtie c'$ then for every $\tbt$-class, the sum of the weights of the edges in $\Dc_{\cH}$ directed to nodes~$c$ and~$c'$, from the nodes in that $\tbt$-class, is the same. Moreover, for $e_i, e_j \in E$, if $e_i\tbowtie e_j$ then there is a bijection between their input sets, $I(e_i)$ and $I(e_j)$, in~$\Dc_{\cH}$ that preserves the $\tbt$-classes. Thus, for the hyperedges $e_i$ and $e_j$ in $\cH$, we have $\overrightarrow{m} (e_i) = \overrightarrow{m} (e_j)$. 
If for two cells~$c$ and~$c'$ of~$C$ we have $c\tbowtie c'$ then for every $\tbt$-class~$K$ we have $I(c) \cap K \ne \emptyset$ if and only if $I(d) \cap K \ne \emptyset$. Thus, in terms of~$\cH$, we have that~$\BS(c)$ has hyperedges with a certain pattern $\overrightarrow{m} (e)$ if and only if $\BS(c')$ also has hyperedges with that pattern $\overrightarrow{m} (e)$. Moreover, as for every $\tbt$-class~$K$ the sum of weights of the edges in $I(c) \cap K \ne \emptyset$ equals the sum of weights of the edges in $I(c') \cap K \ne \emptyset$, we have that the weight 
of each pattern $\overrightarrow{m}(e)$ on the cell~$c$ equals the weight 
of that pattern on the cell~$c'$. Thus, $\tbt_{\cH}$ is balanced. We conclude then that each balanced equivalence relation~$\tbt$ for~$\Dc_{\cH}$ projects into a balanced equivalence relation~$\tbt_{\cH}$ for~$\cH$. 
\end{proof}

There may not be a bijection between the set of balanced equivalence relations for an hypernetwork $(\cH,W)$ and the set of balanced equivalence relations for its incidence digraph $\Dc_{\cH}$. In fact, from  Definition~\ref{def:equiv_rel} and Theorem~\ref{thm:bij_bal}, it follows that if two balanced relations~$\bowtie^1$ and~$\bowtie^2$ for~$\cH$ are not the same then the associated balanced relations~$\bowtie^1_\Dc$ and~$\bowtie^2_\Dc$ for~$\Dc_{\cH}$ are also not the same. Nonetheless, two different balanced relations~$\tbt^1$ and~$\tbt^2$ for~$\Dc_{\cH}$ can project into the same balanced relation $\tbt^1_{\cH} = \tbt^2_{\cH}$ for $\cH$.

\begin{exam} \label{ex:not_bij}
Consider again the directed hypernetwork $\cH$ of Example~\ref{ex:first} on the left of Figure~\ref{figure_example1}. 
The hyperedges of~$\cH$ are 
\begin{align*}
e_1 &= \left( \{2,5\}, \{1\} \right), & e_2 &= \left( \{2\}, \{2,4\} \right), & e_3 &= \left( \{1,2\}, \{6\} \right), \\
e_4 &= \left( \{4,6\}, \{3,5\} \right),& e_5 &= \left( \{4\}, \{3\} \right).&
\end{align*} 
The input equivalence relation for the hypernetwork~$\cH$ is 
\[\simIo = \{ \{1,5,6\},\, \{ 2,4\},\, \{ 3\}\}\]
and the incidence digraph~$\Dc_{\cH}$ for~$\cH$ is shown in   Figure~\ref{B_figure_example1}. 
 
The equivalence relations 
\begin{align*}
\tbt^1 &= \{ \{1, 5, 6\},\ \{2,4\}, \{3\},\ \{e_1, e_3, e_4\}, \ \{e_2\}, \  \{e_5\} \}
\intertext{
and 
}
\tbt^2 &= \{ \{1, 5, 6\},\ \{2,4\}, \{3\},\ \{e_1, e_3, e_4\}, \ \{e_2, e_5\} \}
\end{align*}
for $\Dc_{\cH}$ are balanced and project into the same balanced equivalence relation
\[
\bt = \tbt^1_{\cH} = \tbt^2_{\cH} = \{ \{1, 5, 6\},\ \{2,4\}, \{3\} \}
\]
for~$\cH$. \hfill $\Diamond $
\end{exam}

Nevertheless, it also follows from Definition~\ref{def:equiv_rel} and Theorem~\ref{thm:bij_bal} that the set of balanced equivalence relations for a hypernetwork $(\cH, W)$ can be otained by the projection of the balanced equivalence relations for its incidence digraph $\Dc_{\cH}$.

Let $\Hc=(C,E)$ be a hypergraph with nodes/cells~$C$ and edges~$E$. The balanced relations of a hypernetwork~$(\Hc, W)$ and the digraph~$\Dc_{\cH}=(C(\Dc_{\cH}), E(\Dc_{\cH}))$ associated with the hypergraph~$\Hc$ are related as stated in Theorem~\ref{thm:bij_bal}. How do the synchrony subspaces relate? 
For~$\Dc_{\cH}$ consider cells $C(\Dc_{\cH}) = C\cup E$ equipped with phase space~$\R$; since there are two ``types'' of cells for $\Dc_\cH$, we write~$x_c$ for the state of $c\in C$ and~$x_e$ for the state of $e\in E$. For an equivalence relation $\tbt$ on~$C(\Dc_{\cH})$ for~$\Dc_{\cH}$,  
consider the polydiagonal subspace 
\begin{align*}
\Delta_{\tbt} &= \sset{x_c = x_{c'} \mbox{ if }  c \tbowtie c'
, x_e = x_{e'} \mbox{ if } e \tbowtie e' }.
\intertext{ 
For the projected equivalence relation~$\tbt_{\cH}$ on~$C$ for~$\cH$ obtained from~$\tbt$ consider the usual polydiagonal subspace}
\Delta_{\tbt_{\cH}} &= \sset{x_c = x_{c'} \mbox{ if }  c\tbowtie_{\cH} c' }.
\end{align*}
In terms of synchrony subspaces for the hypernetwork~$(\cH, W)$ we have then the following result.

In terms of synchrony subspaces for the hypernetwork $(\cH, W)$ we have then that they can be obtained via the `projection' of the synchrony subspaces of the adjacency matrix of the incidence digraph $\Dc_{\cH}$.

\begin{thm}
Let  $(\cH, W)$ be a weighted directed hypernetwork and $\Dc_{\cH}$ the associated incidence digraph. Let $\tbt_{\cH}$ and $\tbt$ be equivalence relations and $\Delta_{\tbt_{\cH}}$ and $\Delta_{\tbt}$ polydiagonal subspaces, as defined above.
A polydiagonal subspace $\Delta$ is a synchrony subspace for the hypernetwork $(\cH, W)$ if and only if $\Delta = \Delta_{\tbt_{\cH}}$ with $\Delta_{\tbt}$ a synchrony subspace of the adjacency matrix of the digraph $\Dc_{\cH}$.
\end{thm}
\begin{proof}

Let $\Delta_{\tbt}$ be the polydiagonal subspace associated with an equivalence relation $\tbt$ for the incidence digraph $\Dc_{\cH}$, as defined above. By the definition of balanced relation, $\Delta_{\tbt}$ is a synchrony subspace of (is left invariant by) the adjacency matrix of $\Dc_{\cH}$ if and only if $\tbt$ is balanced. By Theorem~\ref{thm:bij_bal}, the balanced equivalence relations for the hypernetwork $(\cH, W)$ are the projection $\tbt_{\cH}$ of the balanced equivalence relations $\tbt$ for the incidence digraph $\Dc_{\cH}$.
Moreover, by Theorem~\ref{thm:sync_bal}, $\tbt_{\cH}$ is balanced  if and only if the polydiagonal subspace $\Delta_{\tbt_{\cH}}$, as defined above, is a synchony subspace for $(\cH, W)$. The result then follows.
\end{proof}

\begin{rem}
A relevant consequence of the results in this section is that the existing results regarding balanced relations and synchrony spaces for networks can be used to obtain analogous results for hypernetworks. For example, the work of Aldis~\cite{A08} with the description of a polynomial-time algorithm to compute the coarsest balanced equivalence relation of a graph and the work of Aguiar and Dias~\cite{AD14} describing an algorithm to compute the lattice of synchrony subspaces for the adjacency matrix of a network.
\hfill $\Diamond$
\end{rem}

\begin{exam}
Consider again the hypernetwork~$\cH$ on the left of Figure~\ref{figure_example1} of Examples~\ref{ex:first} and~\ref{ex:not_bij}. The  admissible equations  are 
\begin{align*}
\dot{x}_1 &= f(x_1) + Q_2(x_1;x_5,x_2)\\
\dot{x}_2 &= f(x_2) + Q_1(x_2;x_2)\\
\dot{x}_3 &= f(x_3) + Q_1(x_3;x_4) + Q_2(x_3;x_4,x_6)\\
\dot{x}_4 &= f(x_4) + Q_1(x_4;x_2)\\
\dot{x}_5 &= f(x_5) + Q_2(x_5;x_4,x_6) \\
\dot{x}_6 &= f(x_6) + Q_2(x_6;x_1,x_2)
\end{align*}
where  $f:\, V \to V$, $Q_1:\, V^2 \to V$,  $Q_2:\, V^3 \to V$ are smooth functions and ~$Q_2$ is symmetric under permutation of the last two coordinates.
Looking at the equations, we can conclude that the set of nontrivial  synchrony subspaces for the hypernetwork~$\cH$ is given by
\[
\sset{
\Delta_1 = \set{x}{x_2 = x_4},\ 
\Delta_2 = \set{x}{x_1 = x_5 = x_6, x_2 = x_4}
}\, .
\]

Now, let us see how we can get this set of 
synchrony subspaces using the incidence digraph $\Dc_{\cH}$ associated with $\cH$. The digraph $\Dc_{\cH}$ is represented in Figure~\ref{B_figure_example1} 
and its adjacency matrix given by
\[
A_{\small{\Dc_{\cH}}} =
\left[
\begin{array}{c|c}
0_{6 \times 6} & W \\
\hline
T & 0_{5 \times 5}
\end{array}
\right],
\]
with
$$
W = 
\left[
\begin{array}{ccccc}
1 & 0& 0 & 0 & 0  \\
0 & 1& 0 & 0 & 0  \\
0 & 0& 0 & 1 & 1  \\
0 & 1& 0 & 0 & 0  \\
0 & 0& 0 & 1 & 0  \\
0 & 0& 1 & 0 & 0  
\end{array}
\right]
\qquad
\mbox{ and }
\qquad
T = 
\left[
\begin{array}{cccccc}
0 & 1 & 0 & 0 & 1 & 0 \\
0 & 1 & 0 & 0 & 0 & 0 \\
1 & 1 & 0 & 0 & 0 & 0 \\
0 & 0 & 0 & 1 & 0 & 1 \\
0 & 0 & 0 & 1 & 0 & 0 
\end{array}
\right]\, .
$$
For an eigenvalue $\lambda$ of a matrix let~$W_\lambda$ denote the associated (generalized) eigenspace. Moreover, write $\langle v_1, \dotsc, v_k\rangle$ for the span of vectors $v_1, \dotsc, v_k$.
The eigenvalues of the matrix $A_{\small{\Dc_{\cH}}}$ are $\lambda\in\sset{0, \pm 1,\pm 0.5\pm i 0.866}$; the algebraic multiplicity of $\lambda = 0$ is three and that of $\lambda=\pm 1$ is two. The corresponding (generalized) eigenspaces are
\begin{align*}
&W_0 = \langle v_1, v_2,v_3 \rangle, & &W_{-0.5\pm i 0.866} = \langle v_8, v_9\rangle,\\
&W_{-1} = \langle v_4, v_5\rangle, & &W_{0.5 \pm i 0.866} = \langle v_{10}, v_{11}\rangle,\\
&W_{1} = \langle v_6, v_7\rangle,
\end{align*}
where
\begin{align*}
v_1 &= (0,0,1,0,0,0,0,0,0,0,0) & v_2 &= (0,0,1,0,0,0,0,0,0,0,1) \\
v_3 &= (0,0,1,1,0,-1,0,0,0,0,1)  & v_4 &= (1,0,1,0,1,1,-1,0,-1,-1,0) \\
v_5 &= (0,-2,-2,-2,0,0,1,2,1,1,2) & v_6 &= (1,0,1,0,1,1,1,0,1,1,0) \\
v_7 &= (0,2,2,2,0,0,1,2,1,1,2)
\end{align*}
and
\begin{align*}
v_8, v_9 &\in \{ (a,0,b,0,b,c,c,0,b,a,0):\ a \ne b \ne c \in \R \}\\
v_{10}, v_{11} &\in \{ (a,0,b,0,b,c,-c,0,-b,-a,0):\ a \ne b \ne c \in \R \}
\end{align*}

The polydiagonal subspaces given by equalities of cell coordinates and equalities of edge coordinates that are invariant by the adjacency matrix $A_{\small{\Dc_{\cH}}}$ are
\begin{align*}
\tilde \Delta _1 &= \{  x_2 = x_4  \} \\
 &= \langle v_1,v_2\rangle \oplus W_{-1} \oplus W_{1} \oplus W_{-0.5\pm i 0.866} \oplus W_{0.5\pm i 0.866},\\
\tilde \Delta _2 & =  \{x_1 = x_5 = x_6, x_2 = x_4; x_{e_1} = x_{e_3} = x_{e_4} \} \\
& = \langle v_1,v_2\rangle \oplus W_{-1} \oplus W_{1}, \\
\tilde \Delta _3 & =  \{x_1 = x_5=x_6, x_2 = x_4; x_{e_1} = x_{e_3} = x_{e_4}, x_{e_2} = x_{e_5} \} \\
& = \langle v_1\rangle \oplus W_{-1} \oplus W_{1}.
\end{align*}
These now relate to the synchrony spaces of~$\cH$: We have that $\tilde \Delta _1$ `projects into' the synchrony subspace~$\Delta _1$ of~$\cH$ and~$\tilde \Delta _2$ and~$\tilde \Delta _3$ `project into' the synchrony subspace $\Delta _2$ of~$\cH$.
\hfill $\Diamond $
\end{exam}

We stress that our results are valid for both unweighted and weigthed hypernetworks; the previous example can be seen as a hypernetwork where all weights are equal to one. 

\begin{rem} 
Note that there is no need to consider more than one adjacency matrix for the incidence digraph $\Dc_{\cH}$ in order to separate the hyperedges with tails with different multiplicites since those hyperedges as nodes in $\Dc_{\cH}$ cannot synchronize given that the row sum of the corresponding rows in the submatrix $T$ of adjacency matrix $A_{\small{\Dc_{\cH}}}$ is different. 
\hfill $\Diamond$
\end{rem}

\section{Linearization and stability---a case study} 
\label{sec:examples}

In the previous sections, we considered the question what type of synchrony patterns can robustly exist for coupled cell hypernetworks and how they depend on the properties of the underlying hypergraph. We now consider linear stability of solutions on synchrony subspaces; asymptotic stability is crucial to actually observe synchrony patterns in real-world systems. We show that in a class of examples that linear stability may or may not depend on higher-order interactions.

Here we consider weighted directed hypernetworks~$(\Hc,W)$ with~$n$ nodes and directed hyperedges of the two types shown in Figure~\ref{fig:two_types}: There is an edge between nodes $i,j$ with weight~$K_{ij}$ and for each pair of nodes $k,l$ in $\{1, \dotsc, n\}$ there is a hyperedge $\left(\{k,l\}, \{ i\} \right)$ for $i=1, \dotsc, n$ with weight~$H_{kl}$. Note that we do not assume any relationship between the weights~$K_{ij}$ of the pairwise interactions and the weights~$H_{kl}$ between the nonpairwise interactions. For the remainder of this section, we fix a hypernetwork coupling through  the coupling functions 
\begin{align*}
Q_1(p_i; p_j) &= p_i p_j; & Q_2 (p_i; p_k, p_l) &= p_i p_k p_l. 
\end{align*}

\begin{figure}
\begin{center}
\begin{tikzpicture}
 [scale=.15,auto=left, node distance=1.5cm, 
 ]
 \node[fill=white,style={circle,draw}] (n1) at (4,6)  {\small{$k$}};
 \node[fill=white,style={circle,draw}] (n2) at (4,-6) {\small{$l$}};
 \node[fill=white,style={circle,draw}] (n3) at (24,0) {\small{$i$}};
 
 \node [fill=black,style={circle,scale=0.1}] (e1) at (14,0) { };
 \node [fill=black,style={circle,scale=0.1}] (e11) at (16,0) { };
 \node[fill=white,style={circle,draw}] (n4) at (44,0) {\small{$j$}};
  \node[fill=white,style={circle,draw}] (n5) at (58,0) {\small{$i$}};
 \path
        (n1) [-]  edge[bend right=20,thick]  node { } (e1)
        (n2) [-]  edge[bend left=20,thick]  node { } (e1)
         (e1) [->] edge[thick] node {}   (e11)   
        (e11) [-] edge[bend right=0,thick]      node  {{\tiny $H_{kl}$}}  (n3)
         (n4) [->] edge[bend right=0,thick]      node  {{\tiny $K_{ij}$}}  (n5); 
 \end{tikzpicture} 
  \end{center}
 \caption{(Left) A directed hyperedge $e_{kl} = \left(  \{k,l\}, \{ i\} \right)$ with cardinality two tail set and weight $H_{kl}$. (Right) A directed edge $\left(  \{j\}, \{ i\} \right)$ with weight~$K_{ij}$.}\label{fig:two_types}
\end{figure}
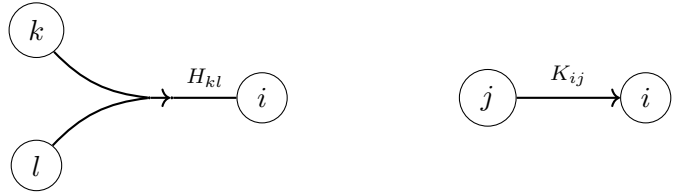

\newcommand{\tr}{\mathsf{T}}

The choice of coupling functions now leads to an admissible coupled cell system for the hypernetwork coupling given by
\begin{align}\dot p_i & =  \left( \sum_{j=1}^{n} K_{ij} p_j - \sum_{k=1}^{n} \sum_{l=1}^{n} H_{kl} p_k p_l \right)p_i
\label{eq:general_rd}
\end{align}
for $i=1, \dotsc, n$ subject to $\sum_{i=1}^{n} p_i = 1$ and $0 \leq p_i \leq 1$.
For a matrix~$A$ let $A^\tr$ denote its transpose.
If we write $K = [K_{ij}]$ and $H = [H_{kl}]$ for the $n \times n$ weight matrices, the system~\eqref{eq:general_rd} can be written in matrix form as
\begin{align}
\dot p_i & = \left( \left( Kp\right)_i - p^\tr H p\right)p_i
\label{eq:general_vdif_rd}
\end{align}
for $i=1, \dotsc, n$.

\begin{rem}
Allesina and Levine~\cite[Supporting Information]{AL11} considered the replicator equations with~$n$ species (see also Hofbauer and Sigmund~\cite{HS98}), that is, equations~\eqref{eq:general_rd} with $K=H$ and~$K$ is skew-symmetric. Here, $K_{ij}$ represents the effect of species~$j$ on the growth rate of species~$i$. The dynamics of species~$i$ is determined by the \emph{fitness of species~$i$} given by $\sum_{j=1}^{n} K_{ij} p_j$ and the \emph{average fitness for the system} $\sum_{k=1}^{n} \sum_{l=1}^{n} K_{kl} p_k p_l$; this ensures that no species can increase in density without other species decreasing. The condition $\sum_{i=1}^{n} p_i = 1$ ensures that total abundance conservation is maintained for all time. In this model
terms of the form $K_{ij}p_ip_j$ represent \emph{pairwise} interactions between the species~$i$ and~$j$ and $\sum_{k=1}^{n} \sum_{l=1}^{n} K_{kl} p_k p_l$ represents an average of \emph{nonpairwise} interactions between all the species.

In~\cite{CT02}, it is shown that for a skew-symmetric $n \times n$ matrix $K$ is skew symmetric the system has a unique equilibrium solution~$p$, which is linearly neutrally stable. For a skew-symmetric matrix~$K$, the quadratic form $w \mapsto w^\tr K w$ is null and with $K=H$ the system~\eqref{eq:general_rd} reduces to
\begin{align}
\dot p_i & = \left( Kp\right)_i p_i
\label{eq:skew_vrd}
\end{align}
for $i=1, \dotsc, n$.
Chawanya and Tokita~\cite{CT02} reports that the condition of skew symmetry of~$K$ (on the interactions between the species) can be used to yield and stabilize a large complex ecosystem.  The antisymmetry model assumption is based on the fact that many species interact with each other in prey-predator or parasitic relationships. 
\hfill $\Diamond$
\end{rem}

We can make the following two observations.
 
\begin{lemma} \label{lem:syn_spa}
(i) The synchrony spaces of~\eqref{eq:general_vdif_rd} are the synchrony spaces of~$K$.\\
(ii) In case $H$ is a \emph{skew symmetric} matrix, that is, $H^\tr = -H$, then the quadratic form $p \mapsto p^\tr H p$ vanishes and 
equations~\eqref{eq:general_vdif_rd} become
\begin{align}
\dot p_i & = \left( \left( Kp\right)_i \right) p_i
\label{eq:skew_general_vdif_rd} 
\end{align}
for $i=1, \dotsc, n$.
\end{lemma}

A straightforward calculation leads to:

\begin{lemma}\label{lem:lin_pnot=0} 
Assume $p$ is an equilibrium of~\eqref{eq:general_vdif_rd} with $p_i \not=0$ for $i=1, \dotsc, n$ and let~$J_p$ denote the Jacobian of~\eqref{eq:general_vdif_rd} at~$p$. Then  
\[
(J_{p})_i  =  \left( (K)_i  - p \left( H + H^\tr \right)  \right) p_i 
\]
for $i=1, \dotsc, n$. Here~$(M)_i$ denotes the $i$th row of the matrix~$M$. Note that the matrix $H + H^\tr$ is always symmetric.
\end{lemma}

We show two examples of system~\eqref{eq:general_vdif_rd}, one with no  nonpairwise interactions and one with nonpairwise interactions, admitting an equilibrium whose stability does depend on the nonpairwise interactions terms. 

\begin{exams} \label{ex:double}
Consider the system~\eqref{eq:general_vdif_rd} where $n =4$ and 
\[
K =  \frac{1}{2} 
\left[
\begin{array}{rrrr}
0 & -1 & 2 & -1 \\
1 & 0 & 0 & -1  \\
-2 & 0 & 0 & 2 \\
1 & 1 & -2 & 0  
\end{array}
\right]
\, .
\]
Note that $K$ is a skew symmetric matrix. The eigenvalues of~$K$ are $\lambda=0$ (double) and a pair of nonzero imaginary eigenvalues $\lambda=\pm i \sqrt{11}/2$. Moreover, 
\[
W_0 = \langle  (1,1,1,1), (0,2,1,0)\rangle\, .
\]
(a) Assume that in~\eqref{eq:general_vdif_rd} there are no nonpairwise interactions, that is, $H = {\bf 0}$.  
We have that $p^* = \frac{1}{4} (1,1,1,1)$ is an equilibrium of the system~\eqref{eq:general_vdif_rd} with stability determined by~$K$ (by  Lemma~\ref{lem:lin_pnot=0}), that is, the equilibrium~$p^*$ has neutral linear stability in the sense that all eigenvalues have zero real part. \\
(b)~Assume now the existence of nonpairwise interactions given by the symmetric matrix 
\[
H =  \left[
\displaystyle  
\begin{array}{rrrr}
2 & -1 & 1 & -2 \\
-1 & 2 & -2 & 1  \\
1 & -2 & 2 & -1 \\
-2 & 1 & -1 & 2  
\end{array}
\right]\, .
\]
Note that~$H$ has eigenvalues $\lambda=0$ (double) and $\lambda=2,\lambda=6$. Moreover, 
\[
W_0 = \langle  (1,1,1,1), (1,0,0,1)\rangle.
\]
We have that $p^* = \frac{1}{4} (1,1,1,1)$ is also an equilibrium of the system~\eqref{eq:general_vdif_rd}. Its (linear) stability is given by Lemma~\ref{lem:lin_pnot=0}.  More precisely, the linear stability of~$p^*$ is determined by 
\[
J_{p^*}  =  \frac{1}{4} \left(K  - \frac{1}{2}  H   \right)  
= 
\displaystyle \frac{1}{8} 
\left[ 
\displaystyle 
\begin{array}{rrrr}
-2 & 0 & 1 & 1 \\
2 & -2 & 2 & -2  \\
-3 & 2 & -2 & 3 \\
3 & 0 & -1 & -2  
\end{array}
\right],
\]
which has a zero eigenvalue, a negative real eigenvalue, and a pair of complex eigenvalues with negative real part. Thus, the equilibrium~$p^*$ is (linearly) stable in the directions transverse to the diagonal~$\langle (1,1,1,1)\rangle$---these are the direction transverse to the synchrony subspace where all cells are synchronized. 
\hfill $\diamond$ 
\end{exams}

Nevertheless, we see next an example where the nonpairwise interactions exist and do not change the stability of the equilibrium. 

\begin{exam}
Consider the system~\eqref{eq:general_vdif_rd} with $n=4$ and 
\[
K = H = 
\left[
\begin{array}{rrrr}
1 & -1 & 1 & -1 \\
1 & 1 & -1 & -1  \\
-1 & 1 & 1 & -1 \\
1 & 1 & 1 & -3  
\end{array}
\right]\, .
\]
Note that $\det(K) =0$ and $\ker(K) = W_0 = \langle (1,1,1,1)\rangle$. Equations~\eqref{eq:general_vdif_rd} evaluate to
\begin{equation} \label{eq:diagonal_pert_eqs}
\dot{p}_i = \left( (K p)_i  - \left( p_1^2 + p_2^2 + p_3^2 -3p_4^2\right) \right) p_i
\end{equation}
for $i=1,2,3,4$.
Although the quadratic form $p \mapsto p^\tr K p = p_1^2 + p_2^2 + p_3^2 -3p_4^2$ is not identically null, it vanishes at $p \in \ker(K)$. 
We have that $p^* = \frac{1}{4}(1,1,1,1)$ is the unique equilibrium~$p$ of system~\eqref{eq:diagonal_pert_eqs} with $p_i >0$ for $i=1, \dotsc, 4$. Note that~$K$ has eigenvalues $\lambda \in \sset{0,-2, 1\pm i \sqrt{3}}$ and $W_{-2} = \langle (1,1,1,3)\rangle$. Thus 
\[
\Delta = \set{p}{p_1 = p_2 = p_3} = W_0 \oplus W_{-2} 
\]
is a synchrony space for~$K$ and thus, by Lemma~\ref{lem:syn_spa}, also for the system~\eqref{eq:diagonal_pert_eqs}. Moreover, 
\[
K + K^\tr = 
\left[
\begin{array}{rrrr}
2 & 0 & 0 & 0 \\
0 & 2 & 0 & 0  \\
0 & 0 & 2 & 0 \\
0 & 0 & 0 & -6  
\end{array}
\right]\, .
\]
By Lemma~\ref{lem:lin_pnot=0}, the linear stability of the equilibrium $p = \frac{1}{4} (1,1,1,1)$ of the system~\eqref{eq:diagonal_pert_eqs} is determined by the Jacobian matrix
\[
J_p = \displaystyle \frac{1}{4} 
\left( K - \frac{1}{4} 
\left[
\begin{array}{rrrr}
2 & 2 & 2 & -6 \\
2 & 2 & 2 & -6 \\
2 & 2 & 2 & -6 \\
2 & 2 & 2 & -6 
\end{array}
\right]\
 \right)= 
 \frac{1}{16} 
 \left[
\begin{array}{rrrr}
2 & -6 & 2 & 2 \\
2 & 2 & -6 & 2 \\
-6 & 2 & 2 & 2 \\
2 & 2 & 2 & -6 
\end{array}
\right],
\]
which has eigenvalues $0, -\frac{1}{2}$, and $\frac{1}{4}(1\pm i \sqrt{3})$. That is, it has the same stability as for the system without nonpairwise 
interactions, $H={\bf 0}$. \hfill $\Diamond$ 
\end{exam}

\section{Discussion} \label{sec:discussion}

\begin{figure}
\begin{tikzpicture}
 [scale=.15,auto=left, node distance=1.5cm, 
 ]
  \node[fill=white,style={circle,draw}] (n9) at (-36,9)  {\small{9}};
 \node[fill=blue!70,style={circle,draw}] (n10) at (-36,0) {\tiny{10}};
 \node[fill=blue!70,style={circle,draw}] (n11) at (-36,-9)  {\tiny{11}};
 \node[fill=pink,style={circle,draw}] (n12) at (-24,0) {\tiny{12}};
 \node[fill=white,style={circle,draw}] (n8) at (-6,-10)  {\small{8}};
 \node[fill=white,style={circle,draw}] (n1) at (-12,9)  {\small{1}};
 \node[fill=white,style={circle,draw}] (n2) at (-12,0) {\small{2}};
 \node[fill=blue!70,style={circle,draw}] (n3) at (-12,-9)  {\small{3}};
 \node[fill=pink,style={circle,draw}] (n4) at (0,0) {\small{4}};
  \node[fill=blue!70,style={circle,draw}] (n5) at (12,9)  {\small{5}};
 \node[fill=blue!70,style={circle,draw}] (n6) at (12,0)  {\small{6}};
  \node[fill=blue!70,style={circle,draw}] (n7) at (12,-9)  {\small{7}};
\path
        (n9) [->]  edge[bend right=10,thick] node[above] {{\tiny $1$}}   (n12)
        (n10) [->]  edge[thick]   node[above] {{\tiny $1$}}   (n12)
         (n11) [->]  edge[bend left=10,thick]  node[above] {{\tiny $1$}}  (n12) 
        (n1) [->]  edge[bend left=10,thick] node[above] {{\tiny $1$}}  (n12)
        (n2) [->]  edge[thick]  node[above] {{\tiny $1$}} (n12)
         (n3) [->]  edge[bend right=10,thick]  node[above] {{\tiny $1$}} (n12)          
        (n1) [->]  edge[bend right=20,thick] node[above] {{\tiny $1$}}  (n4)
        (n2) [->]  edge[bend right=20, thick]  node[above] {{\tiny $1$}}   (n4)
         (n8) [->]  edge[bend left=20,thick]  node[above] {{\tiny $1$}}  (n4)       
         (n5) [->]  edge[bend left=10,thick] node[above] {{\tiny $1$}}  (n4)
         (n6) [->]  edge[bend right=10,thick]  node[above] {{\tiny $1$}}  (n4)
         (n7) [->]  edge[bend right=10,thick]  node[above] {{\tiny $1$}}  (n4)  
;
 \end{tikzpicture}  
 \caption{The equivalence relation with three classes represented by the three colours is balanced for the network.}\label{hyp_more_notbal}
 \end{figure}
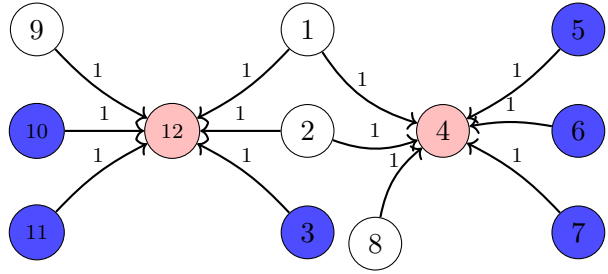

Here we developed a framework for coupled cell systems with higher-order interactions. In contrast to other approaches to dynamics on hypergraphs---including~\cite{Mulas2020,Salova2021}---our framework allows for directionality of the interactions and coupling weights. 
The framework is restricted by the assumption of homogeneity in the $k$th order coupling: The interaction is mediated by a single coupling function~$Q_k$ for any edge of tail size~$k$. These assumptions do shape the set of admissible vector fields. Recall the hypernetwork of
Example~\ref{example_b_notb}(ii), which is depicted in Figure~\ref{hyp_notbal}. As an example, the admissible evolution equations for nodes~$4$ and~$12$ take the shape
\begin{align*}
\dot{x}_4 &= f(x_4)+Q_3(x_4;x_1,x_2,x_8)+Q_3(x_4;x_5,x_6,x_7),\\
\dot{x}_{12} &= f(x_{12})+Q_3(x_{12};x_1,x_2,x_3)+Q_3(x_{12};x_9,x_{10},x_{11})\, .
\end{align*}
By contrast, if we forget the hyperedge structure and consider the related network shown in Figure~\ref{hyp_more_notbal} then the equations for cells~$4$ and~$12$ in the formalism of  Golubitsky, Stewart and collaborators~\cite{SGP03, GST05} have the form
\begin{align*}
\dot{x}_4 &= g(x_4; {x_1, x_2, x_5, x_6, x_7, x_8}),\\
\dot{x}_{12} &= g(x_{12}; {x_1, x_2, x_3, x_9, x_{10}, x_{11}}),
\end{align*}
where~$g$ is invariant under permutations of the last six arguments.
Even though the combinatorial representation of the equations is a network (a directed graph), the admissible vector fields that are determined by the interaction function~$g$ can have nonlinear dependencies between the cell coordinates~$x_k$.
By contrast, in the additive input setup~\cite{F15, BF17, ADF17, AD18}  
no nonlinear interactions beyond pairs of cells are possible and the admissible equations for cells~$4$ and~$14$ have the form
\begin{align*}
\dot{x}_4 &= f(x_4) + h(x_4, x_1) + h(x_4, x_2) + h(x_4,  x_5)\\&\qquad\quad + h(x_4, x_6) + h(x_4, x_7) + h(x_4, x_8), \\
\dot{x}_{12} &= f(x_{12}) + h(x_{12}, x_1) + h(x_{12}, x_2) + h(x_{12}, x_3)\\&\qquad\quad + h(x_{12}, x_9) + h(x_{12}, x_{10})+ h(x_{12}, x_{11}).
\end{align*}
The admissible vector fields of our framework are richer than the additive setup. Moreover, they explicitly capture higher-order interaction structure, which is only implicit in the classical formalism of Golubitsky, Stewart, and collaborators but important from a dynamical point of view; cf.~Section~\ref{sec:examples}.

What is an appropriate combinatorial structure to encode higher-order interactions in network dynamical systems (cf.~\cite{Bick2021})? The framework developed above is phrased in terms of (directed) hypergraphs. 
First, the hypergraphs employed are nonstandard: The tails of each hyperedge is a multiset rather than a set. This is crucial to define a quotient of a hypernetwork 
without making further assumptions on the coupling functions as arguments on the synchrony subspace can appear multiple times.
Second, different hypergraphs can represent the same coupled cell hypernetwork. This is due to the fact that hyperedge-heads can contain more than one element 
{which may allow} to easily identify symmetries (cf.~Proposition~\ref{prop:FullSym}).

It is worth pointing out that in the formalism developed above we typically consider all admissible vector fields at the same time. More specifically, we ask: What are the dynamical features of all ordinary differential equations (ODE) that are compatible with the hypernetwork structure? This elucidates the constraints network structure imposes. For example, Theorem~\ref{thm:sync_bal} allows to translate structural properties (balanced relations on a hypergraph) into dynamical properties (any ODE consistent with the hypernetwork will have a particular synchrony subspace). Consequently, these properties are not specific to any choice of coupling function.
While this is the same approach as in traditional coupled cell systems, the approach is in contrast to some applications where a fixed coupling function is considered: A specific coupling function may be imposed by a particular physical system. But a nongeneric choice of coupling function can lead to nongeneric dynamical behavior and nonproper hypernetwork couplings (Definition~\ref{def:Proper}).

The importance of higher-order interactions in network dynamical systems has repeatedly been highlighted.
The framework presented here bridges coupled cell systems and higher-order interaction networks. Specifically, it allows to characterize synchrony patterns (whether global or localized/clustered). While other approaches are possible, our framework strikes a balance between generality and results that can elucidate synchronization phenomena in real-world systems.

\vspace{5mm}

\noindent {\bf Acknowledgments.}
MA and AD were partially supported by CMUP (UID/MAT/00144/2013), which is funded by FCT (Portugal) with national (MEC) and European structural funds (FEDER), under the partnership agreement PT2020. 
CB acknowledges support from the Engineering and Physical Sciences Research Council (EPSRC) through the grant EP/T013613/1.

\vspace{5mm}

\bibliographystyle{unsrt}
\bibliography{biblio} 

\begin{thebibliography}{10}

\bibitem{SGP03}
Ian Stewart, Martin Golubitsky, and Marcus Pivato.
\newblock {Symmetry Groupoids and Patterns of Synchrony in Coupled Cell
  Networks}.
\newblock {\em SIAM Journal on Applied Dynamical Systems}, 2(4):609--646, 2003.

\bibitem{GST05}
Martin Golubitsky, Ian Stewart, and Andrei T{\"{o}}r{\"{o}}k.
\newblock {Patterns of Synchrony in Coupled Cell Networks with Multiple
  Arrows}.
\newblock {\em SIAM Journal on Applied Dynamical Systems}, 4(1):78--100, 2005.

\bibitem{F05}
Michael~J. Field.
\newblock {Combinatorial dynamics}.
\newblock {\em Dynamical Systems}, 19(3):217--243, 2004.

\bibitem{F15}
Michael~J. Field.
\newblock {Heteroclinic Networks in Homogeneous and Heterogeneous Identical
  Cell Systems}.
\newblock {\em Journal of Nonlinear Science}, 25(3):779--813, 2015.

\bibitem{BF17}
Christian Bick and Michael~J. Field.
\newblock {Asynchronous networks and event driven dynamics}.
\newblock {\em Nonlinearity}, 30(2):558--594, 2017.

\bibitem{ADF17}
Manuela A.~D. Aguiar, Ana Paula~S. Dias, and Flora Ferreira.
\newblock {Patterns of synchrony for feed-forward and auto-regulation
  feed-forward neural networks}.
\newblock {\em Chaos}, 27(1):013103, 2017.

\bibitem{AD18}
Manuela A.~D. Aguiar and Ana Paula~S. Dias.
\newblock {Synchronization and equitable partitions in weighted networks}.
\newblock {\em Chaos}, 28(7):073105, 2018.

\bibitem{Battiston2020}
Federico Battiston, Giulia Cencetti, Iacopo Iacopini, Vito Latora, Maxime
  Lucas, Alice Patania, Jean-gabriel Young, and Giovanni Petri.
\newblock {Networks beyond pairwise interactions: Structure and dynamics}.
\newblock {\em Physics Reports}, 874:1--92, 2020.

\bibitem{Bick2021}
Christian Bick, Elizabeth Gross, Heather~A. Harrington, and Michael~T. Schaub.
\newblock {What are higher-order networks?}
\newblock {\em arXiv:2104.11329}, 2021.

\bibitem{AR16}
Peter Ashwin and Ana Rodrigues.
\newblock {Hopf normal form with $S_N$ symmetry and reduction to systems of
  nonlinearly coupled phase oscillators}.
\newblock {\em Physica D}, 325:14--24, 2016.

\bibitem{BAR16}
Christian Bick, Peter Ashwin, and Ana Rodrigues.
\newblock {Chaos in generically coupled phase oscillator networks with
  nonpairwise interactions}.
\newblock {\em Chaos}, 26(9):094814, 2016.

\bibitem{AL17}
Giorgio Ausiello and Luigi Laura.
\newblock {Directed hypergraphs: Introduction and fundamental algorithms—A
  survey}.
\newblock {\em Theoretical Computer Science}, 658:293--306, 2017.

\bibitem{JI07}
Jeffrey~H. Johnson and Pejman Iravani.
\newblock {The multilevel hypernetwork dynamics of complex systems of robot
  soccer agents}.
\newblock {\em ACM Transactions on Autonomous and Adaptive Systems}, 2(2):5,
  2007.

\bibitem{KHZ14}
Soo-Jin Kim, Jung-Woo Ha, and Byoung-Tak Zhang.
\newblock {Bayesian evolutionary hypergraph learning for predicting cancer
  clinical outcomes}.
\newblock {\em Journal of Biomedical Informatics}, 49:101--111, 2014.

\bibitem{J16}
J.H. Johnson.
\newblock {Hypernetworks: Multidimensional relationships in multilevel
  systems}.
\newblock {\em The European Physical Journal Special Topics},
  225(6-7):1037--1052, 2016.

\bibitem{S12}
Francesco Sorrentino.
\newblock {Synchronization of hypernetworks of coupled dynamical systems}.
\newblock {\em New Journal of Physics}, 14(3):033035, 2012.

\bibitem{Mulas2020}
Raffaella Mulas, Christian Kuehn, and J{\"{u}}rgen Jost.
\newblock {Coupled dynamics on hypergraphs: Master stability of steady states
  and synchronization}.
\newblock {\em Physical Review E}, 101(6):062313, 2020.

\bibitem{Salova2021a}
Anastasiya Salova and Raissa~M. D'Souza.
\newblock {Cluster synchronization on hypergraphs}.
\newblock {\em arXiv:2101.05464}, pages 1--7, 2021.

\bibitem{Salova2021}
Anastasiya Salova and Raissa~M. D'Souza.
\newblock {Analyzing states beyond full synchronization on hypergraphs requires
  methods beyond projected networks}.
\newblock {\em arXiv:2107.13712}, pages 1--17, 2021.

\bibitem{GLP93}
Giorgio Gallo, Giustino Longo, Stefano Pallottino, and Sang Nguyen.
\newblock {Directed hypergraphs and applications}.
\newblock {\em Discrete Applied Mathematics}, 42(2-3):177--201, 1993.

\bibitem{AS2021}
Arguello and Stadler.
\newblock {Whitney’s connectivity inequalities for directed hypergraphs}.
\newblock {\em The Art of Discrete and Applied Mathematics}, pages 1--14, 2021.

\bibitem{S07}
Ian Stewart.
\newblock The lattice of balanced equivalence relations of a coupled cell
  network.
\newblock {\em Mathematical Proceedings of the Cambridge Philosophical
  Society}, 143(1):165--183, 2007.

\bibitem{A08}
John~W. Aldis.
\newblock A polynomial time algorithm to determine maximal balanced equivalence
  relations.
\newblock {\em International Journal of Bifurcation and Chaos in Applied
  Sciences and Engineering}, 18(2):407--427, 2008.

\bibitem{AD14}
Manuela A.~D. Aguiar and Ana Paula~S. Dias.
\newblock {The Lattice of Synchrony Subspaces of a Coupled Cell Network:
  Characterization and Computation Algorithm.}
\newblock {\em Journal of Nonlinear Science}, 6(24):949--996, 2014.

\bibitem{AL11}
Stefano Allesina and Jonathan~M. Levine.
\newblock A competitive network theory of species diversity.
\newblock {\em Proceedings of the National Academy of Sciences},
  108(14):5638--5642, 2011.

\bibitem{HS98}
Josef Hofbauer and Karl Sigmund.
\newblock {\em Evolutionary Games and Population Dynamics}.
\newblock Cambridge University Press, 1998.

\bibitem{CT02}
Tsuyoshi Chawanya and Kei Tokita.
\newblock Large-dimensional replicator equations with antisymmetric random
  interactions.
\newblock {\em Journal of the Physical Society of Japan}, 71(2):429--431, 2002.

\end{thebibliography}

\end{document}